\documentclass[11pt]{amsart}

\usepackage{stmaryrd}
\usepackage{amscd,amsxtra,amssymb,mathrsfs, bbm}

\usepackage{multirow}
\usepackage[pdftex]{hyperref}
\usepackage{longtable}
\usepackage{tikz}
\usepackage[utf8]{inputenc}
\usepackage{float}
\usepackage{color}
\usetikzlibrary{calc, positioning, shapes, fit, matrix, decorations}
\usetikzlibrary{decorations.shapes}

\usepackage[cmtip]{xypic}
\usepackage[all]{xy}
\xyoption{arc}

\newtheorem{theorem}{Theorem}[section]
\newtheorem{corollary}[theorem]{Corollary}
\newtheorem{lemma}[theorem]{Lemma}
\newtheorem{proposition}[theorem]{Proposition}

\theoremstyle{definition}
\newtheorem{definition}[theorem]{Definition}
\newtheorem{remark}[theorem]{Remark}
\newtheorem{example}[theorem]{Example}

\numberwithin{equation}{section}

\DeclareMathAlphabet{\mathpzc}{OT1}{pzc}{m}{it}

\DeclareMathOperator{\TF}{\mathsf{TF}}

\DeclareMathOperator{\rk}{rk}

\renewcommand{\ker}{\mathsf{Ker}}

\renewcommand{\dim}{\mathsf{dim}}

\DeclareMathOperator{\Sem}{\mathsf{Sem}}
\DeclareMathOperator{\Coh}{\mathsf{Coh}}

\DeclareMathOperator{\Pic}{\mathsf{Pic}}
\DeclareMathOperator{\MM}{\mathbb{M}}

\DeclareMathOperator{\krdim}{\mathsf{kr.dim}}

\DeclareMathOperator{\depth}{\mathsf{depth}}

\DeclareMathOperator{\ord}{\mathsf{ord}}

\DeclareMathOperator{\Hom}{\mathsf{Hom}}

\DeclareMathOperator{\Tor}{\mathsf{Tor}}

\DeclareMathOperator{\Ext}{\mathsf{Ext}}

\DeclareMathOperator{\GL}{\mathsf{GL}}
\DeclareMathOperator{\Aut}{\mathsf{Aut}}

\DeclareMathOperator{\End}{\mathsf{End}}
\DeclareMathOperator{\Mat}{\mathsf{Mat}}
\DeclareMathOperator{\Spec}{\mathsf{Spec}}

\DeclareMathOperator{\BL}{\mathsf{BL}}

\input xy
\xyoption{all}

\def\bu{{\scriptscriptsize\bullet}}

\newcommand{\kk}{\mathbbm{k}}

\newcommand{\llbrace}{(\!(}
\newcommand{\rrbrace}{)\!)}

\newcommand{\FF}{\mathbb{F}}
\newcommand{\GG}{\mathbb{G}}

\newcommand{\JJ}{\mathbb{J}}
\renewcommand{\mod}{\mathsf{mod}}
\def\bu{{\scriptscriptstyle\bullet}}

\newcommand{\kA}{\mathcal{A}}
\newcommand{\kB}{\mathcal{B}}
\newcommand{\kE}{\mathcal{E}}
\newcommand{\kF}{\mathcal{F}}
\newcommand{\kG}{\mathcal{G}}
\newcommand{\kH}{\mathcal{H}}
\newcommand{\kI}{\mathcal{I}}
\newcommand{\kO}{\mathcal{O}}
\newcommand{\kL}{\mathcal{L}}
\newcommand{\kP}{\mathcal{P}}
\newcommand{\kK}{\mathcal{K}}

\newcommand{\kV}{\mathcal{V}}

\newcommand{\kT}{\mathcal{T}}

\newcommand{\lar}{\longrightarrow}

\newcommand{\gB}{\mathfrak{B}}
\newcommand{\gQ}{\mathfrak{Q}}

\newcommand{\gD}{\mathfrak{D}}
\newcommand{\gE}{\mathfrak{E}}
\newcommand{\gS}{\mathfrak{S}}

\newcommand{\gK}{\mathfrak{K}}

\newcommand{\CC}{\mathbb C}

\newcommand{\idm}{\mathfrak{m}}

\newcommand{\EE}{\mathbb{E}}

\newcommand{\ZZ}{\mathbb{Z}}

\newcommand{\DD}{\mathbb{D}}
\newcommand{\PP}{\mathbb{P}}

\newcommand{\TT}{\mathbb{T}}

	\def\kJ{\mathcal J}

			\def\bu{\bullet}

\def\Mat{\mathop\mathrm{Mat}}

\def\8{\infty}			
	\def\+{\oplus}		
\def\*{\otimes}

\def\kA{\mathcal A} 
\def\kB{\mathcal B} \def\kO{\mathcal O}
\def\kC{\mathcal C} \def\kP{\mathcal P}
 
\def\kE{\mathcal E} 
\def\kF{\mathcal F} \def\kS{\mathcal S}
\def\kG{\mathcal G} \def\kT{\mathcal T}
\def\kH{\mathcal H} \def\kU{\mathcal U}
\def\kI{\mathcal I} \def\kV{\mathcal V}
\def\kJ{\mathcal J} 
\def\kK{\mathcal K} 
\def\kL{\mathcal L}

	\def\NN{\mathbb N}

\def\DMO{\DeclareMathOperator}
\DMO{\ob}{Ob}            \DMO{\mor}{Mor}
\DMO{\Ker}{Ker}
\DMO{\id}{Id}

\title[Fourier--Mukai transform  and commuting  differential operators]{Fourier--Mukai transform on Weierstrass cubics and commuting differential operators}

\author{Igor Burban}
\address{
Universit\"at Paderborn,
Institut f\"ur Mathematik,
Warburger Strasse 100,
33098 Paderborn,
Germany
}
\email{burban@math.uni-paderborn.de}

\author{Alexander Zheglov}
\address{
Moscow State University,
Faculty of Mechanics and Mathematics, Leninskie gory, GSP-1, Moscow, 119899, Russian Federation
}
\email{azheglov@math.msu.su}

\begin{document}

\begin{abstract} In this article, we describe the spectral sheaves of algebras of commuting differential operators of genus one and rank two with singular spectral curve, solving a problem posed by Previato and Wilson. We also classify all indecomposable semi--stable  sheaves of slope one and  ranks two or  three on a cuspidal Weierstra\ss{} cubic.
\end{abstract}

\maketitle

The purpose of this  article is  to study spectral sheaves of genus one commutative subalgebras in the algebra
of ordinary differential operators   $\gD = \CC\llbracket z\rrbracket[\partial]$.

Let  $\Lambda \subset \CC$ be a lattice and $\wp(z)$ be the corresponding Weierstra\ss{} function. As it was observed by Wallenberg \cite{Wallenberg} in 1903, the ordinary differential operators
\begin{equation}\label{E:Wallenberg}
P = \partial^2 - 2 \wp(z+\alpha) \quad \mbox{and} \quad
Q = 2 \partial^3 - 6 \wp(z+\alpha) \partial - 3 \wp'(z+\alpha),
\end{equation}
commute for all $\alpha \in \CC$ and obey the relation
$
Q^2 = 4P^3 - g_2 P  - g_3,
$
where $g_2$ and $g_3$ are the Weierstra\ss{} parameters  of the lattice $\Lambda$, see  \cite{Wallenberg}.

In 1968 Dixmier   discovered another interesting example \cite{Dixmier}: for any $\kappa \in \CC$,
put $D := \partial^2 + z^3 + \kappa$ and consider
\begin{equation}\label{E:Dixmier}
P = D^2 + 2z \quad \mbox{and} \quad
Q = D^3 + \frac{3}{2}\bigl(z D + D z\bigr).
\end{equation}
Then $P$ and $Q$  commute and satisfy the relation
$
Q^2 = P^3 - \kappa.$ Dixmier also shown that the subalgebra $\CC[P, Q] \subset \mathfrak{D}$  is in fact \emph{maximal}.

It turns out that
any non--trivial commutative subalgebra $\gB$ in $\gD$ is finitely generated
and has Krull dimension one. Moreover, the affine curve $X_0 = \mathsf{Spec}(\gB)$ admits a one--point compactification by a \emph{smooth} point $p$ to a projective curve $X$. The arithmetic genus of $X$ is called \emph{genus} of the algebra $\gB$. Additionally, the algebra $\gB$ determines
a coherent torsion free sheaf $\kF$ on the curve $X$ having  the following characteristic properties:
\begin{itemize}
\item For any point $q \in X_0$ (smooth or singular) corresponding to an algebra homomorphism
$\gB \stackrel{\chi}\lar \CC$, we have an isomorphism of vector spaces
$$
\kF\bigl|_{q}^\ast \lar \bigl\{
f \in \CC\llbracket z\rrbracket \,|\, P \circ f = \chi(P)f \; \mbox{\rm for all}\;  P \in \gB
\bigr\}.
$$
\item The evaluation map $H^0(X, \kF) \stackrel{\mathsf{ev}_p}\lar \kF\bigl|_{p}$ is an isomorphism
and $H^1(X, \kF) = 0$.
\end{itemize}
The curve $X$ (respectively, the sheaf $\kF$)  is called \emph{spectral curve} (respectively, \emph{spectral sheaf}) of the algebra $\gB$. The rank of the torsion free sheaf $\kF$ is called \emph{rank} of $\gB$. Krichever correspondence \cite{Krichever} asserts that any non--trivial commutative subalgebra of $\gB$ of \emph{rank one} is essentially determined by its spectral data
$(X, p,  \kF)$. The description of commutative subalgebras of $\gD$ of higher rank is more complicated. It was first given  by Krichever \cite{Krichever76, Krichever77, Krichever} and then elaborated by many authors,  including Drinfeld \cite{Drinfeld},  Mumford \cite{Mumford}, Segal and Wilson \cite{SegalWilson}, Verdier \cite{Verdier}, Mulase \cite{Mulase1} and  others.

A first description of genus one and rank two commutative subalgebras of $\gD$ was obtained by Krichever and
Novikov \cite{KN}, who also discovered a connection between this kind of problems and  soliton solutions of certain non--linear PDE equations. In their Ansatz, however, the spectral curve $X$ was taken to be smooth. Since that time, the study
of genus one commutative subalgebras of $\gD$ attracted a considerable attention, see for example  \cite{Grinevich, GrinevichNovikov, Dehornoy, PrW, Mokhov}. We refer to \cite[Section 1]{PrW} for an illuminative overview.  It is not difficult to show that for any (normalized) genus one and rank two commutative subalgebra $\gB \subset \gD$ there exist two operators $L, M \in \gB$ such that   $\gB = \CC[L, M]$ and
\begin{equation}\label{E:AnsatzNew}
L = \partial^4 + a_2 \partial^2 + a_1 \partial + a_0,  \quad M = 2  L^{\frac{3}{2}}_+,\quad M^2 = 4 L^3 - g_2 L - g_3
\end{equation}
for some $g_2, g_3 \in \CC$, see \cite{Grun, PrW} or Proposition \ref{P:RankTwo} below.
Here, $L^{\frac{3}{2}}$ is taken in the algebra
of pseudo--differential operators $\CC\llbracket z\rrbracket\llbrace \partial^{-1}\rrbrace$ and $L^{\frac{3}{2}}_+$ is the  projection of $L^{\frac{3}{2}}$ onto $\gD$.  A full description of all  operators $L$ as in (\ref{E:AnsatzNew})
 satisfying the constraint $[L, M] = 0$  for $M = 2  L^{\frac{3}{2}}_+$ was obtained  by Gr\"unbaum \cite{Grun}, who also got convenient formulae for the coefficients $a_0, a_1$ and $a_2$.
In this article we deal with the following

\smallskip
\noindent
\textbf{Problem}. What is the spectral sheaf $\kF$ of the algebra $\gB = \CC[L, P]$ of genus one and rank two, expressed through the coefficients $a_0, a_1$ and $a_2$ from (\ref{E:AnsatzNew}) in the case when the spectral curve $X$ is \emph{singular}? In particular, what is the spectral sheaf of  Dixmier's family (\ref{E:Dixmier}) for $\kappa = 0$?

Previato and Wilson gave a complete solution of the above problem in  the case the spectral curve $X$ is smooth \cite[Theorem 1.2]{PrW}. Their answer was given in terms of Gr\"unbaum's parameters \cite{Grun} as well as of Atiyah's classification
of vector bundles on elliptic curves \cite{At}. The description of  the spectral sheaf  in the case of a singular spectral curve  was left as an open problem. Quoting \cite[Page 109]{PrW}: ``We have not worked out   the case when the curve $X$ is singular, that is, when $X$ is nodal or cuspidal cubic. It would probably be rather complicated (because of the need to consider torsion free sheaves)''.

 It turns out that the problem of Previato and Wilson can be completely solved thanks to the technique of derived categories and  Fourier--Mukai transforms on the Weierstra\ss{} cubics
\cite{BK1}. The main idea is that instead of dealing with the spectral sheaf $\kF$ directly, it is easier to  describe its   Fourier--Mukai transform $\kT$, which is a certain torsion sheaf on $X$ defined through the canonical short exact sequence:
$$
0 \lar \Gamma(X, \kF) \otimes \kO \stackrel{\mathsf{ev}}\lar \kF \lar \kT \lar 0.
$$
It turns out that at least the support of $\kT$ can be algorithmically computed. Moreover, since the length of $\kT$ is two in Previato--Wilson problem, one can determine the isomorphism class of $\kT$ using deformation arguments. The key point is the following: as the arithmetic genus of $X$ is equal to one,
the spectral sheaf $\kF$ can be recovered from $\kT$ via the inverse Fourier--Mukai transform.
This approach brings a new light on the method of \cite{PrW} and allows to treat the analogous problem for genus one commutative subalgebras of $\gD$ of arbitrary  rank.

The structure of this article is the following. In Section \ref{S:KrichTh}, we review the theory of commutative subalgebras in $\gD$. The major new  results of this section  are Theorem \ref{T:axiomSSh} giving an ``axiomatic description''  of the spectral sheaf of a commutative subalgebra  $\gB \subset \gD$ as well as  Theorem \ref{T:STtwist} explaining the appearance of Fourier--Mukai transforms in  Krichever's theory.

A classification of indecomposable coherent sheaves   on a smooth elliptic curve was obtained by Atiyah in \cite{At}.  In this case, indecomposable vector bundles are automatically semi--stable. Semi--stable torsion free sheaves
 of integral slope on a nodal Weierstra\ss{} cubic were explicitly classified in \cite{BK1}, see also \cite{DrozdGreuel, Survey}. On the other hand, the category of semi--stable torsion free sheaves of slope one on a cuspidal cubic curve turns out to be representation wild \cite{Drozd, Survey}. Nevertheless, one can obtain a full classification of all rank two or three semi--stable coherent sheaves of slope one on a cuspidal cubic curve. This is done in Section \ref{S:SST}, see in particular Theorem \ref{T:classificationRankTwo} (providing a self--contained classification in  the nodal case as well) and Corollary
\ref{C:ListRankThree}.

In Section \ref{S:SpecSheaves} we give a full answer on the question of Previato and Wilson \cite{PrW}, describing the spectral sheaf of a genus one and rank two commutative subalgebra of $\gD$ with  singular spectral curve, see Theorem  \ref{T:selfadj}, Theorem \ref{T:notLF} and Theorem \ref{T:interestingfamily}. In particular, we describe all such commutative subalgebras, whose spectral sheaf is indecomposable and not locally free, see Corollary \ref{C:OperatorsNotLF}.
Finally, taking  the Fourier transform of  Dixmier's example (\ref{E:Dixmier}), we illustrate  how the spectral sheaf of a genus one and rank three commutative subalgebra of $\gD$ can be explicitly determined, see Example \ref{E:exRank3}. We hope that a more detailed treatment  of the action of automorphisms of the  Weyl algebra $\mathfrak{W} = \CC[z][\partial]$ on the spectral sheaves of genus one commutative subalgebras of $\mathfrak{W}$
 would be of interest  for  various studies related with  Dixmier's conjecture about $\Aut(\mathfrak{W})$, see \cite{MironovZheglov}.

\smallskip
\noindent
\textit{Acknowledgement}. Parts of this work were done  at the
Mathematical Research Institute in Oberwolfach within the ``Research in Pairs'' programme in the period
October 5  -- October 17, 2015, as well as during research stays of the second--named author at the University of Cologne. The research of the second--named author was supported by RFBR grants
16-01-00378-a and 16-51-55012 China-a. We are also grateful to Emma Previato for  fruitful discussions.

\smallskip
\noindent
\textit{List of notations}.
Since this work uses  quite different techniques, for convenience of the reader we introduce now  the most important  notations used in this paper.

\smallskip
\noindent
1.~In what follows, $\gD = \CC\llbracket z\rrbracket[\partial]$ is the algebra of ordinary differential operators, whose coefficients are formal power series. Next,
$\gE = \CC\llbracket z\rrbracket \llbrace\partial^{-1}\rrbrace$ is the algebra of ordinary pseudo--differential operators and $\mathfrak{W} = \CC[z][\partial]$ is the Weyl algebra. Finally, $\gB$ will always denote a commutative subalgebra of $\gD$. Then $X_0 = \Spec(\gB)$ is  the affine spectral curve of $\gB$ and $F = \CC[\partial]$ is  the spectral module of $\gB$; $(X, p, \kF)$ stands for  the spectral datum of $\gB$ (the spectral curve, point at infinity and the spectral sheaf).

\smallskip
\noindent
2.~In Section \ref{S:SpecSheaves}, a description of rank two and genus one commutative subalgebra $\gB \subset \gD$ is given in terms
 of Gr\"unbaum's parameters  $K_{10}, K_{11}, K_{12}, K_{14} \in \CC$ and $f \in \CC\llbracket z\rrbracket$ \cite{Grun};  $L \in \gB$ is  a normalized operator of order four, whereas $M = 2 L^{\frac{3}{2}}_+$ is  another generator of $\gB$ of order six, thus $\gB = \CC[L, M]$.

\smallskip
\noindent
3.~For a (projective) curve $X$ (which is not necessarily the spectral curve of a commutative subalgebra of $\gD$), $\Coh(X)$ denotes the category of coherent sheaves on $X$, $\Tor(X)$ is its full subcategory of torsion sheaves, $\TF(X)$ is the category of torsion free sheaves on $X$  and $D^b\bigl(\Coh(X)\bigr)$ is the bounded derived category of $\Coh(X)$.

\smallskip
\noindent
4.~In Sections \ref{S:SST} and \ref{S:SpecSheaves},  $X = X_{g_2, g_3}= \overline{V(y^2 - 4 x^3 + g_2 x + g_3)} \subset \mathbb{P}^2$
 is  a Weierstra\ss{} cubic curve with parameters  $g_2, g_3 \in \CC$, $p = (0:1:0)$ is the ``infinite point'' of $X$; if $X$ is singular then $s = (0:0:1)$ denotes the unique singular point of $X$. Next,   $\Sem(X)$ is  the category of semi--stable coherent sheaves on $X$ of slope one. The functor $$\TT: D^b\bigl(\Coh(X)\bigr) \lar D^b\bigl(\Coh(X)\bigr)$$ is the Fourier--Mukai transform with the kernel $\kI_\Delta[1]$ (the shifted ideal sheaf of the diagonal). It induces an equivalence of abelian categories $\FF: \Sem(X) \lar \Tor(X)$; $\GG$ will denote a quasi--inverse functor to $\FF$. In these terms, a classifications of  rank  two objects of $\Sem(X)$ is given: $\kS$ is the unique rank one object of $\Sem(X)$ which is not locally free, $\kA$ is the rank two Atiyah sheaf; if $X$ is singular and   $q \in X$ is a smooth point, then  $\kB_q$ is the (uniquely determined) indecomposable rank two locally free sheaf from $\Sem(X)$, whose determinant is $\kO\bigl([p]+[q])$ and whose Fourier--Mukai transform is supported at $s$. If $X$ is cuspidal then $\kU$ is the unique indecomposable object of $\Sem(X)$ of rank two which is not locally free; there are two such objects $\kU_\pm$ in the case $X$ is nodal.

\section{Commutative subalgebras in the algebra of  differential operators}\label{S:KrichTh}
Let $\gD = \CC\llbracket z\rrbracket[\partial] = \bigl\{\sum\limits_{i = 0}^n a_i(z) \partial^i \,|\, a_i(z) \in \CC\llbracket z\rrbracket, 0 \le i \le n\bigr\}$ be the algebra of ordinary differential operators with coefficients in the algebra $\CC\llbracket z\rrbracket$ of formal power series. In this section we shall review the theory of commutative subalgebras of $\gD$.  The first systematic study of this problem  dates back to a work of Schur \cite{Schur}. Burchnall and Chaundy \cite{BurchnallChaundy, BurchnallChaundy2, BurchnallChaundy3} and Baker \cite{Baker} obtained a full classification of pairs of commuting differential operators of coprime orders.
The modern algebro--geometric treatment of arbitrary commutative subalgebras in $\gD$ was  initiated by Krichever \cite{Krichever76, Krichever77, Krichever}. This theory has been extensively applied by Novikov and his school in  the study of soliton  solutions
of various non--linear partial differential   equations, see for example the survey \cite{KN}. Krichever's approach was formalized and further developed by Drinfeld \cite{Drinfeld}, Mumford \cite{Mumford}, Verdier \cite{Verdier}, Segal and Wilson \cite{SegalWilson}
 and Mulase \cite{Mulase1}.  The literature dedicated to this area is vast and the described bibliography is definitely uncomplete. There are numerous survey articles on this subject, see for example
 \cite{Previato,Wilson,Mulase2}. Nonetheless, for our purposes we felt it was necessary to review this theory once again, setting the notation and introducing all the relevant notions. The major  novelties  of  this section  are Theorem \ref{T:axiomSSh} giving an axiomatic description of the spectral sheaf of a commutative subalgebra of $\gD$ and  Theorem \ref{T:STtwist} explaining the appearance
 of derived categories in Krichever's theory.

\subsection{Some elementary properties of the algebra $\gD$}
 Let us begin with the following well--known result about automorphisms of $\gD$.
\begin{lemma}\label{E:autom}
Let $\varphi$ be a non-zero algebra endomorphism of $\gD$. Then there exist $u \in \CC\llbracket z\rrbracket$ satisfying
$u(0) = 0$ and $u'(0) \ne 0$, and $v \in \CC\llbracket z\rrbracket$ such that
\begin{equation}\label{E:automorphisms}
\left\{
\begin{array}{ccc}
z & \stackrel{\varphi}\mapsto & u \\
\partial & \stackrel{\varphi}\mapsto & \dfrac{1}{u'} \partial + v.
\end{array}
\right.
\end{equation}
In particular, $\varphi$ is an automorphism of $\gD$, i.e.~$\End(\gD)\backslash \{0\} = \Aut(\gD)$.
\end{lemma}
\begin{proof}
Let $u := \varphi(z) \in \gD$. It is not difficult to show that $u$ belongs $\CC\llbracket z\rrbracket$ and satisfies the  properties stated  in the theorem. Let  $P := \varphi(\partial) = a_n \partial^n + a_{n-1} \partial^{n-1} + \dots + a_0 \in \gD$ for some
$n \in \NN$, where $a_n \ne 0$. Clearly, $[P, u] = n u' a_n \partial^{n-1} + \mbox{l.o.t}$, hence $[\partial, z] = 1 = [P, u]$ if and only if $n = 1$ and $a_1 = \dfrac{1}{u'}$.
\end{proof}

\begin{remark} Let $w \in \CC\llbracket z\rrbracket$ be a unit (i.e.~$w(0) \ne 0$). Then for the inner automorphism
$\mathsf{Ad}_w: \gD \lar \gD, \; P \mapsto w^{-1} P w$, we have:
$$
\left\{
\begin{array}{ccl}
z & \mapsto & z \\
\partial & \mapsto & \partial + \dfrac{w'}{w}.
\end{array}
\right.
$$
Note that for any $\CC\llbracket z\rrbracket \ni v = \sum\limits_{i = 0}^\infty \beta_i z^i = \beta_0 + \tilde{v}$, the formal power series $w := \exp(v) = e^{\beta_0} \exp({\tilde{v}})$ is a unit in $\CC\llbracket z\rrbracket$. Therefore,
any automorphism $\varphi \in \Aut(\gD)$ satisfying $\varphi(z) = z$ is inner, see (\ref{E:automorphisms})
\end{remark}

\begin{proposition}\label{P:normalizing} Let $P = a_n \partial^n + a_{n-1} \partial^{n-1} + \dots + a_0 \in \gD$, where
$a_n(0) \ne 0$. Then there exists $\varphi \in \Aut(\gD)$ such that
\begin{equation}\label{E:normalized}
Q := \varphi(P) =  \partial^n + b_{n-2}\partial^{n-2} + \dots + b_0
\end{equation}
for some $b_0,\dots, b_{n-2} \in \CC\llbracket z\rrbracket$. Moreover, if $Q \in \gD$ is a \emph{normalized}  differential operator of positive order (i.e.~a differential operator having the form (\ref{E:normalized}))
and $\psi$  an inner automorphism of $\gD$ such that $\psi(Q) = Q$ then $\psi = \mathrm{id}$.
\end{proposition}
\begin{proof}
By assumption, $a_n$ is a unit in $\CC\llbracket z\rrbracket$. Therefore, there exists $a \in \CC\llbracket z\rrbracket$ such that $a^n = a_n$. It implies that $P = \bigl(a\partial\bigr)^n + \mbox{l.o.t}$. Hence, there exists a change of variables transforming $P$ into an operator of the form $\widetilde{P}:= \partial^n + c_{n-1}\partial^{n-1} + \dots + c_0$. Applying now to $\widetilde{P}$ an automorphism (\ref{E:automorphisms}) with $u = z$ and $v = - \dfrac{c_{n-1}}{n}$, we get a normalized operator $Q$. This proves the first statement. The proof
of the second statement is straightforward.
\end{proof}

\begin{definition}
A differential operator $P = a_n \partial^n + a_{n-1} \partial^{n-1} + \dots + a_0 \in \gD$ of positive order $n$ is called \emph{formally elliptic} if $a_n \in \CC^*$.
\end{definition}

\smallskip
\noindent
The following useful observation is due to Verdier \cite[Lemme 1]{Verdier}.
\begin{lemma}\label{L:Verdier}
Let $\gB$ be a commutative subalgebra of $\gD$ containing a formally elliptic element $P$. Then \emph{all} elements of
$\gB$ are formally elliptic.
\end{lemma}

\begin{remark} An algebra $\gB \subset \gD$ containing a formally elliptic element is called \emph{elliptic}.
There exists non--trivial non--elliptic commutative subalgebras in $\gD$, i.e.~tho\-se which are not of the form
$\CC[P]$, where $P$ is a non--elliptic operator. Nonetheless, the major interest concerns those commutative subalgebras of $\gD$ which belong to the subalgebra $\CC\{z\}[\partial]$ of ordinary differential operators, whose coefficients are \emph{convergent} power series. If $P = a_n \partial^n + a_{n-1} \partial^{n-1} + \dots + a_0$ is such an operator then shifting the variable $z \mapsto z + \varepsilon$ with $\varepsilon \in \CC$ such that $|\varepsilon|$ is  sufficiently small, we may always achieve that $a_n(0) \ne 0$. Note that this operation can not be extended on the whole $\gD$. Still,  one can show that all elements of $\gB$ belong to  $\CC\{z\}[\partial]$  (this follows for example from Schur's theorem \cite[Theorem 2.2]{Mulase2}, see for example \cite[Lemma 5.3]{Mulase1}) and one can choose a common radius of convergence for all coefficients of all elements of $\gB$. According to Proposition \ref{P:normalizing}, we can transform $P$ into a normalized formally elliptic differential operator. Therefore, in the sequel all commutative subalgebras
of $\gD$ are assumed
\begin{itemize}
\item to contain an elliptic operator of positive order (i.e.~being elliptic)
\item to be \emph{normalized}, meaning that all elements of $\gB$ of minimal positive order are  normalized.
\end{itemize}
The last assumption eliminates redundant degrees of freedom   in the problem of classification of commutative subalgebras of differential operators: if $\gB \subset \gD$ is a  normalized elliptic subalgebra and $\varphi$ an inner automorphism of $\gD$ such that $\varphi(\gB) = \gB$ then $\varphi = \mathrm{id}$.
\end{remark}

\subsection{Spectral curve and spectral module of commuting  differential operators}
\begin{definition}\label{D:rank}
Let $\gB$ be a commutative subalgebra of $\gD$. We call the natural number
$$
r = \mathrm{rk}(\gB) = \mathrm{gcd}\left\{\mathsf{ord}(P) \big| P \in \gB \right\}
$$
the \emph{rank} of $\gB$.
\end{definition}

\begin{theorem}\label{T:speccurve}
Let $\gB$ be a commutative subalgebra of $\gD$.
\begin{enumerate}
\item Then $\gB$ is finitely generated integral domain of Krull dimension one. In particular, $\gB$ determines
an integral affine algebraic curve $X_0:= \Spec(\gB)$.
\item Moreover, $X_0$ can be compactified to a projective algebraic curve $X$ by adding  a \emph{single smooth} point $p$, which is determined by the valuation
$$
\mathsf{val}_p: \gQ \lar \ZZ, \quad \frac{P}{Q} \mapsto
\frac{\mathsf{ord}(Q) - \mathsf{ord}(P)}{r},
$$
where $\gQ$ is the quotient field of $\gB$ and $r$ is the rank of $\gB$.
\end{enumerate}
\end{theorem}

\noindent
\emph{Comment to the proof}. Algebraic curves entered for the first time into the theory of commutative subalgebras of $\gD$  in the works
of Burchnall and Chaundy \cite{BurchnallChaundy, BurchnallChaundy2} and in a  greater generality in the works
of Krichever \cite{Krichever76, Krichever77}. In the stated form, this result can be found
in the article of Mumford \cite[Section 2]{Mumford} (see also Verdier \cite[Proposition 1]{Verdier} and \cite[Theorem 3.3]{Mulase1}). The spectral curve  $X$ is  defined as follows. For any $i \in \NN$ denote
$$
\gB_i := \gB \cap \gD_{\le ir} = \bigl\{P \in \gB \,\big|\, \ord(P) \le ir \bigr\}.
$$
Let $\widetilde\gB = \bigoplus\limits_{i = 0}^\infty \gB_i t^i \subset \gB[t]$ be the Rees algebra of $\gB$. Then we put $X = \mathsf{Proj}(\widetilde{\gB})$, see \cite[Section 2.3]{EGAII}. The principal ideal $(t) \subset \widetilde\gB$ is a prime ideal, since the graded algebra
$$
\overline{\gB}:= \widetilde{\gB}/(t) \cong \bigoplus\limits_{i = 0}^\infty \bigl(\gB_{i}/\gB_{i-1}\bigr)
$$
is obviously a domain. It can be shown that $\krdim(\overline{\gB}) = 1$. Therefore, $(t)$  defines a point of $X$, which is the ``infinite'' point $p$. The same consideration also shows that $\depth(\widetilde{\gB}) = \depth(\overline{\gB}) + 1 = 2$, hence
the graded algebra $\widetilde{\gB}$ is \emph{Cohen--Macaulay}. See also
\cite[Theorem 2.1]{KOZ} for an elaboration of Mumford's approach  as well as for a  generalization on the  higher--dimensional cases.

\begin{definition}
The projective curve $X = X_0 \cup \{p\}$ is called \emph{spectral curve} of a commutative subalgebra $\gB \subset
\gD$. The arithmetic genus of $X$ is called \emph{genus} of $\gB$.
\end{definition}

\begin{example}
In the example of Wallenberg (\ref{E:Wallenberg}), the algebra $\CC[P, Q]$ has rank one and genus one. In the example of Dixmier (\ref{E:Dixmier}), the algebra $\CC[P, Q]$ has rank two  and genus one for any $\kappa \in \CC$.
\end{example}

\begin{definition} Let $\gB \subset \gD$ be a commutative subalgebra.
Consider the right $\gD$--module $F := \gD/z\gD \stackrel{\cong}\lar \CC[\partial], \, \overline{a(z) \partial^n} \mapsto a(0) \partial^n$. Clearly, the \emph{right} action of $\gD$ on $\CC[\partial]$ satisfies  the following rules:
\begin{equation}\label{E:action}
\left\{
\begin{array}{ccl}
p(\partial) \diamond \partial & = & \partial \cdot p(\partial) \\
p(\partial) \diamond z  & = &  p'(\partial).
\end{array}
\right.
\end{equation}
Restricting the action (\ref{E:action}) on the subalgebra $\gB$, we endow $F$ with the structure of a $\gB$--module. Since the algebra $\gB$ is commutative, we shall view $F$ as a \emph{left} $\gB$--module (although having the natural right action
in mind).
\end{definition}

\begin{theorem}
Let $\gB \subset \gD$ be a commutative subalgebra of rank $r$. Then  $F$ is finitely generated and torsion free over $\gB$. Moreover,
$
\gQ \otimes_{\gB} F \cong \gQ^{\oplus r},
$
i.e.~$\mathrm{rk}_{\gB}(F) = \mathrm{rk}(\gB)$. In other words, the rank of the algebra $\gB$ in the sense of Definition \ref{D:rank} coincides
with the rank of $F$ viewed as a $\gB$-module.
\end{theorem}

\begin{proof} In the stated form, this  result can be found in \cite[Proposition 3]{Verdier} and \cite[Section 2]{Mumford}. See also
\cite[Theorem 2.1]{KOZ} for another treatment as well as for a  generalization on the  higher--dimensional cases.  Because some ideas  the  proof
will be used later, we provide  its details here.

\smallskip
\noindent
Since $r | \ord(P)$ for any $P \in \gB$, it is easy to see   that the elements $1, \partial, \dots, \partial^{r-1}$ of $F$ are linearly independent  over $\gB$. Let ${F}^\circ := \langle 1, \partial, \dots, \partial^{r-1}\rangle_\gB \subset F$. It is sufficient to prove that the quotient
 $F/{F}^\circ$ is finite dimensional over $\CC$. Let
 $
 \Sigma := \bigl\{d \in \NN_{0} \, \big| \, \mbox{there exists}\, P \in \gB \; \mbox{with} \; \ord(P) = d \bigr\}.
 $
 Obviously, $\Sigma$ is a sub--semi--group of $r\NN_0$. Moreover, one can find  $l \in \NN$ such that for all $m \ge l$ there exists some element
 $P_m \in \gB$ such that $\ord(P_m) = mr$. One can easily prove that
 $
 F/\widetilde{F}$ is spanned over $\CC$ by the classes of $1, \partial, \dots,  \partial^{lr},
 $
 hence $ \gQ \otimes_{\gB} F \cong \gQ \otimes_{\gB} {F}^\circ \cong \gQ^{\oplus r}$.
\end{proof}

\smallskip
\noindent
 Recall that according to
the Nullstellensatz, the points of $X_0$ stand in bijection with the algebra homomorphisms $\gB \lar \CC$
(called in what follows \emph{characters}).
\begin{definition}
Let $q \in X_0$ be any point  and $\chi= \chi_q: \gB \lar \CC$ the corresponding character. We call the $\CC$--vector space
\begin{equation}\label{E:SolSpace}
\mathsf{Sol}\bigl(\gB, \chi\bigr):= \bigl\{f\in \CC\llbracket z\rrbracket \big| P\circ f = \chi(P) f \; \mbox{for all}\; P \in \gB\}
\end{equation}
the \emph{solution space} of the algebra $\gB$ at the point $q$. Here,  we apply the usual left action $\circ$ of $\gD$ on $\CC\llbracket z\rrbracket$. Observe, that $\mathsf{Sol}\bigl(\gB, \chi\bigr)$ has  a natural $\gB$--module structure.
\end{definition}

\smallskip
\noindent
The geometric meaning of the $\gB$--module $F$ is explained by the next result.
\begin{theorem} The following $\CC$--linear map
\begin{equation}
F \stackrel{\eta_\chi}\lar \mathsf{Sol}\bigl(\gB, \chi\bigr)^\ast, \quad \partial^i \mapsto \Bigl(f \mapsto \frac{1}{i!}f^{(i)}(0)\Bigr)
\end{equation}
is also $\gB$--linear, where $\mathsf{Sol}\bigl(\gB, \chi\bigr)^\ast = \Hom_{\CC}\bigl(\mathsf{Sol}\bigl(\gB, \chi\bigr), \CC\bigr)$ is the vector space dual of the solution space. Moreover, the induced map
\begin{equation}\label{E:fiberspecsheaf}
\gB/\ker(\chi) \otimes_\gB F \stackrel{\bar{\eta}_\chi}\lar \mathsf{Sol}\bigl(\gB, \chi\bigr)^\ast
\end{equation}
is an isomorphism of $\gB$--modules.
\end{theorem}

\begin{proof} These statements can be found in \cite[Section 2]{Mumford} or \cite[Proposition 5]{Verdier}, where the proofs are briefly outlined.  Since this result plays a central role in our work, we  give a detailed proof here. First note that the following map
\begin{equation}
\Hom_{\CC}\bigl(F, \CC\bigr) \stackrel{\Phi}\lar \CC\llbracket z \rrbracket, \quad \lambda \mapsto \sum\limits_{p= 0}^\infty \frac{1}{p!} \lambda(\partial^p) z^p
\end{equation}
is  an  isomorphism of left $\gD$--modules.
Let $\gB \stackrel{\chi}\lar \CC$ be a character, then $\CC = \CC_\chi:= \gB/\ker(\chi)$ is a left $\gB$--module. We obtain a $\gB$--linear map
\begin{equation}
\Psi: \Hom_{\gB}(F, \CC_\chi) \stackrel{I}\lar \Hom_{\CC}(F, \CC) \stackrel{\Phi}\lar \CC\llbracket z \rrbracket,
\end{equation}
where $I$ is the forgetful map. The image of $I$ consists of those $\CC$--linear functionals, which are also $\gB$--linear, i.e.
\begin{equation*}
\mathsf{Im}(I) = \bigl\{\lambda \in \Hom_{\CC}(F, \CC) \; \big| \;  \lambda(P \diamond  \,-\,) = \chi(P)\cdot \lambda(\,-\,)\; \mbox{for all} \; P \in \gB\bigr\}.
\end{equation*}
This implies that $\mathsf{Im}(\Psi) = \mathsf{Sol}(\gB, \chi)$. Next, we have  a canonical  isomorphism of $\gB$--modules:
$
\Hom_{\gB}(F, \CC_\chi) \cong \Hom_{\CC}\bigl(\gB/\ker(\chi) \otimes_{\gB} F, \CC\bigr).
$
 Dualizing again, we get an isomorphism of vector spaces
$$
\Psi^\ast: \mathsf{Sol}(\gB, \chi)^\ast \lar \bigl(\gB/\ker(\chi) \otimes_{\gB} F\bigr)^{\ast\ast} \cong \gB/\ker(\chi) \otimes_{\gB} F.
$$
It remains to observe that $\Psi^\ast$ is also $\gB$--linear and  $\bigl(\Psi^\ast\bigr)^{-1} = \bar{\eta}_\chi$.
\end{proof}

\begin{remark} The isomorphism (\ref{E:fiberspecsheaf}) has  the following geometric meaning: if we view $F$ as a coherent sheaf on $X_0 = \Spec(A)$ then for any point $q \in X_0$ (smooth or singular)  we have:
$F\big|_{q} \cong \mathsf{Sol}(\gB, \chi)^\ast$, where $\gB \stackrel{\chi}\lar \CC$ is the character corresponding to the point $q$. Because of this fact, $F$ is called \emph{spectral module} of the algebra $\gB$.
\end{remark}

\begin{corollary}\label{C:DimSolSpace}
Let $\gB \subset \gD$ be a commutative subalgebra of rank $r$. Then for any character
$\gB \stackrel{\chi}\lar \CC$ we have:
$
r \le \dim_{\CC}\bigl(\mathsf{Sol}(\gB, \chi)\bigr) < \infty.
$
Moreover, $\dim_{\CC}\bigl(\mathsf{Sol}(\gB, \chi)\bigr) \ge r+1$ if  only if $\chi$ defines a \emph{singular point} $q \in X_0$  \emph{and} $F$ is \emph{not locally free} at $q$.
\end{corollary}

\subsection{Axiomatic description of the spectral sheaf} Let $\gB \subset \gD$ be a commutative subalgebra   and  $F = \CC[\partial]$ be its spectral module.  According to Theorem \ref{T:speccurve}, the affine curve $X_0 = \mathsf{Spec}(\kB)$ admits a canonical compactification $X = X_0 \cup \{p\}$. It turns out that  the spectral module  $F$ can also be \emph{canonically} extended from $X_0$ on the whole projective curve  $X$. The following result implicitly existed in the literature, although we are not aware of any  reference for a direct proof. However, since it plays very important role in our paper, we provide full details now.

\begin{theorem}\label{T:axiomSSh} Let $\gB \subset \gD$ be a rank $r$ commutative subalgebra and  $H = \bigl\langle 1, \partial, \dots, \partial^{r-1}\bigr\rangle_{\CC}$. Then the following results are true.
\begin{enumerate}
\item There exists a pair $(\kF, \varphi)$, where $\kF$ is a torsion free coherent sheaf on $X$ and $\Gamma(X_0, \kF) \stackrel{\varphi}\lar F$ an isomorphism
of $\gB$--modules (here we use an identification $\gB \cong \Gamma(X_0, \kO)$)  inducing an isomorphism of vector spaces
$\Gamma(X, \kF) \stackrel{\varphi_{\mid}}\lar H$. In particular,  the following diagram of vector spaces
$$
\xymatrix{
\Gamma(X, \kF) \ar@{^{(}->}[rr]^{\imath}  \ar[d]_-{\varphi_{\mid}} & & \Gamma(X_0, \kF) \ar[d]^-{\varphi}\\
H  \ar@{^{(}->}[rr]&& F
}
$$
is commutative (the restriction map $\imath$ is injective since the coherent sheaf $\kF$ is assumed to be torsion free).
\item Let  $(\kF', \varphi')$ be  another pair satisfying the properties
of the previous paragraph. Then there exists an isomorphism  $\kF \stackrel{\psi}\lar \kF'$  making the following diagram
$$
\xymatrix{
\Gamma(X_0, \kF) \ar[rr]^-{\Gamma(X_0, \psi)} \ar[rd]_{\varphi}& & \Gamma(X_0, \kF') \ar[ld]^{\varphi'}\\
 & F &
}
$$
commutative. In other words, the pair $(\kF, \varphi)$ is unique up to an automorphism of $\kF$. The torsion free sheaf $\kF$ is called \emph{spectral sheaf} of $\gB$.
\item The spectral sheaf $\kF$ has the following additional properties: the evaluation map $\Gamma(X, \kF) \stackrel{\mathsf{ev}_p}\lar \kF\bigl|_{p}$ is an isomorphism
and $H^1(X, \kF) = 0$.
\end{enumerate}
\end{theorem}

\begin{proof} We divide the proof into the following logical steps.

\smallskip
\noindent
\underline{Step 1} (Beauville--Laszlo triples). Let us introduce the following notation.
\begin{itemize}
\item $\widehat{O}_p$ is the completion of the local ring $\kO_p$ and  $\widehat{Q}_p$ is the field of fractions of $\widehat{O}_p$.
\item $\Gamma(X_0, \kO) \stackrel{l_p}\lar \widehat{Q}_p$ is the map assigning to a regular function on $X_0$ (viewed as a rational function on $X$) its Laurent expansion at the point $p$.
\end{itemize}
Then we obtain the following Cartesian diagram in the category of schemes:
\begin{equation}\label{E:BLdiagram}
\begin{array}{c}
\xymatrix{
\Spec\bigl(\widehat{Q}_p\bigr) \ar[rr]^-{\nu} \ar[d]_-{\zeta} & & \Spec\bigl(\widehat{O}_p\bigr) \ar[d]^-{\xi} \\
X_0 \ar[rr]^-{\eta} & & X
}
\end{array}
\end{equation}
where all morphisms $\xi, \zeta,  \eta, \nu$ are the canonical ones (in particular, the morphism $\zeta$ is defined by the algebra homomorphism $l_p$). The category $\BL(X)$  is defined as follows. Its objects are triples $(G, V, \tau)$ (called \emph{BL--triples}), where
\begin{itemize}
\item $G$ is a finitely generated torsion free $\gB$--module (which will be  also viewed as a torsion free coherent sheaf on the affine spectral curve
$X_0 = \Spec(\gB)$),
\item $V$ is a free $\widehat{O}_p$--module (viewed as a locally free sheaf on $\Spec(\widehat{O}_p)$),
\item $\zeta^* G \stackrel{\tau}\lar \nu^* V$ is an isomorphism of coherent sheaves  on the affine scheme $\Spec(\widehat{Q}_p)$.
\end{itemize}
The definition of morphisms in the category $\BL(X)$ is straightforward.

\smallskip
\noindent
Then the following results are true.

\begin{itemize}
\item The functor $\TF(X) \lar \BL(X)$, assigning to a torsion free sheaf $\kF$ the BL--triple $(\eta^*\kF, \xi^*\kF, \tau_\kF)$ is an equivalence of categories (here, $\zeta^* \bigl(\eta^* \kF\bigr) \stackrel{\tau_\kF}\lar \nu^* \bigl(\xi^* \kF\bigr)$ is the canonical isomorphism), see \cite{BL}.
\item The following sequence of vector spaces is exact (see e.g.~\cite[Proposition 3]{Parshin}):
\begin{equation}\label{E:MVsequence}
0 \lar \Gamma(X, \kF) \lar \Gamma(X_0, \kF) \oplus \widehat{\kF}_p \lar Q\bigl(\widehat{\kF}_p\bigr) \lar H^1(X, \kF) \lar 0.
\end{equation}
Here, $\widehat{\kF}_p = \xi^*(\kF)$, $Q\bigl(\widehat{\kF}_p\bigr) = \nu^\ast\bigl(\widehat{\kF}_p\bigr)$ and all maps in (\ref{E:MVsequence}) are the canonical ones.
\end{itemize}
These results imply that the pair $(\kF, \varphi)$ can be constructed in terms of BL--triples.

\medskip
\noindent
\underline{Step 2} (Beauville--Laszlo triples revisited). In order to simplify the treatment of  the category $\BL(X)$, we give now its alternative description. We introduce the following notation.
\begin{itemize}
\item  $\gE = \CC\llbracket z\rrbracket\llbrace \partial^{-1}\rrbrace$ denotes  the algebra of ordinary pseudo--differential operators.
\item $\gS = \bigl\{1 + \sum\limits_{i = 1}^\infty s_i(z) \partial^{-i}\bigr\} \subset \gE$ is the so--called Volterra--group.
\end{itemize}
According to  Schur's  theory of ordinary pseudo--differential operators, there exists an element $S \in \gS$ (called \emph{Schur operator} of $\gB$) such that $A : = S^{-1} \gB S \subset  \CC\llbrace\partial^{-r}\rrbrace \subset \gE$, see \cite[Proposition 3.1]{Mulase1} (actually, such an operator $S$ is unique only up to a  multiple $S \mapsto ST$ with an appropriate \emph{admissible operator} $T$, see \cite[Definition 4.3]{Mulase1}; however, this non--uniqueness  of the choice of $S$ does not play any role in the sequel). Since the affine spectral curve  $\Spec(\gB)$ can be completed by adding a single smooth point $p$, one can show the  following

\smallskip
\noindent
\underline{Fact}. Let $\gB \stackrel{\alpha}\lar \Gamma(X_0, \kO)$ be a fixed isomorphism of $\CC$--algebras. Then there exists a unique isomorphism
of $\CC$--algebras $\CC\llbrace \partial^{-r}\rrbrace \stackrel{\beta}\lar \widehat{Q}_p$ making  the following diagram
\begin{equation}\label{E:BLpseudodiff}
\begin{array}{c}
\xymatrix{
\Gamma(X_0, \kO) \ar@{^{(}->}[rr]^-{l_p} & & \widehat{Q}_p & & \widehat{O}_p \ar@{_{(}->}[ll] \\
\gB \ar@{^{(}->}[rr]^-{\mathsf{Ad}_S} \ar[rr] \ar[u]^-\alpha &  & \CC\llbrace \partial^{-r}\rrbrace  \ar[u]_-\beta & & \ar@{_{(}->}[ll] \CC\llbracket \partial^{-r}\rrbracket \ar[u]_-{\beta_\mid}
}
\end{array}
\end{equation}
 commutative. Here, $\mathsf{Ad}_S(P) = S^{-1} P S$ for any $P \in \gB$. Note that the map $\beta$ automatically restricts to  an algebra isomorphism
$\CC\llbracket \partial^{-r}\rrbracket \stackrel{\beta_\mid}\lar \widehat{O}_p$. Diagram (\ref{E:BLdiagram}) allows one to rewrite the definition
of the category  $\BL(X)$ in terms, which are more convenient for our purposes.

\smallskip
\noindent
\underline{Step 3} (Spectral sheaf via Beauville--Laszlo triples). We introduce some  new notation.
\begin{itemize}
\item $\widetilde{Q} = \CC\llbrace \partial^{-1}\rrbrace$,  $\widehat{Q} = \CC\llbrace \partial^{-r}\rrbrace$ and $\widehat{O} =  \CC\llbracket \partial^{-r}\rrbracket$.
\item For any $i \in \NN$, let $\nabla_i := \partial^{i} \diamond S \in \widetilde{Q}$ (note that $\ord(\nabla_i) = i$).
\item $W := F \diamond S = \bigl\langle \nabla_i \, \big| \, i \in \NN_0\bigr\rangle_{\CC}$ and ${W}^\circ := {F}^\circ \diamond S = \bigl\langle \nabla_i \, \big| \, 0 \le i \le r-1\bigr\rangle_{A}$.
\item Finally,   $K := H \diamond S = \bigl\langle \nabla_i \, \big| \, 0 \le i \le r-1\bigr\rangle_{\CC}$ and $U = \partial^{r-1} \CC\llbracket \partial^{-1}\rrbracket$.
\end{itemize}
Note that $W$ is a torsion free finitely generated $A$--module (in the terminology of Mulase's work \cite{Mulase1}, $(A, W)$ is a \emph{Schur pair}) and
$U$ is a free $\widehat{O}$--module of rank $r$ (with generators $\nabla_0, \dots, \nabla_{r-1}$). Now we can define an isomorphism of $\widehat{Q}$--vector spaces $W \otimes_{A} \widehat{Q} \stackrel{\tau}\lar U \otimes_{\widehat{O}} \widehat{Q}$ requiring  commutativity of the following diagram:
\begin{equation}
\begin{array}{c}
\xymatrix{
W \otimes_{A} \widehat{Q} \ar[d]_-{\tau}& & \ar[ll]_-{\cong} {W}^\circ \otimes_{A} \widehat{Q} \ar[rr]^-{\mathsf{mult}} & & \widetilde{Q} \ar[d]^{=} \\
U \otimes_{\widehat{O}} \widehat{Q} \ar[rrrr]^-{\mathsf{mult}} & & & & \widetilde{Q}.
}
\end{array}
\end{equation}
From all what was said above, we conclude the following results:
\begin{itemize}
\item $(W, U, \tau)$ is a BL--triple.
\item  $W \cap U  = K$ and $W + U = \widetilde{Q}$ ($W$ and $U$ are identified with their images in $\widetilde{Q}$).
\end{itemize}
Let $\kF$ be the torsion free sheaf on $X$ determined by the BL--triple $(W, U, \tau)$, then we have:
$$
\dim_{\CC}\bigl(\Gamma(X, \kF)\bigr) = r \quad \mbox{and} \quad H^1(X, \kF) = 0.
$$
 Together with the torsion free sheaf $\kF$ defined by the BL--triple $(W, U, \tau)$,  we also  get an isomorphism $\kF\big|_{X_0} \stackrel{\varphi}\lar W$ identifying the space $\Gamma(X, \kF)$ of global sections of $\kF$ with the vector space $K$. Moreover, in the commutative diagram
 $$
 \xymatrix
 { & \Gamma(X, \kF) \ar[ld]_-{\mathsf{ev}'_p} \ar[rd]^-{\mathsf{ev}_p} & \\
 \widehat{\kF}_p \ar[rr] & & \kF\big|_{p}
 }
 $$
we have: $\mathsf{Im}\bigl(\mathsf{ev}'_p\bigr) = \bigl\langle \nabla_i \, \big| \, 0 \le i \le r-1\bigr\rangle_{\CC}$ (here we identify $\widehat{\kF}_p$ with $U$). This implies that the linear map $\mathsf{ev}_p$ is an isomorphism.

\medskip
\noindent
\underline{Step 4} (Uniqueness of the pair $(\kF, \varphi)$). Assume $(\kF', \varphi')$ is an another pair, as in the statement of the theorem. Then we have another BL--triple $(W, U', \tau')$, where $U' \subset \widetilde{Q}$ is a free $\widehat{O}$--module of rank $r$ such that $W \cap U' = K$.
Hence, $\nabla_0, \dots, \nabla_{r-1} \in U'$ implying that
$$
 \bigl\langle \nabla_i \, \big| \, 0 \le i \le r-1\bigr\rangle_{\widehat{O}} = \partial^{r-1} \CC\llbracket \partial^{-1}\rrbracket =:U  \subseteq U'.
$$
Assume that $U' \ne U$. Then there exists some element $\nabla \in U'$ with $d = \ord(\nabla) \ge r$. Next, we can find scalars $\alpha_r, \alpha_{r+1}, \dots, \alpha_d \in \CC$ such that
$$
\widetilde{\nabla} := \nabla - \alpha_d \nabla_d - \dots - \alpha_r \nabla_r \in   \partial^{r-1} \CC\llbracket \partial^{-1}\rrbracket.
$$
This implies that
$
\Delta := \nabla - \widetilde{\nabla} = \alpha_r \nabla_r + \dots + \alpha_d \nabla_d \in W \cap U'.
$
On the other hand, $\ord(\Delta) \ge r$, hence $\Delta \notin K$. Contradiction.
\end{proof}

\smallskip
\noindent
The next result shows that the  axiomatic description of the spectral sheaf $\kF$ given in Theorem \ref{T:axiomSSh}, coincides with the one given in the spirit of Mumford's approach \cite{Mumford}.
\begin{proposition} Let $\gB \subset \gD$ be a commutative subalgebra of rank $r$, $F = \CC[\partial]$ be its spectral module. For any $i \in \NN_0$, we put
$F_i := \CC[\partial]_{< r(i+1)}$. Let $\kF$ be the sheafification of the Rees module
$
\widetilde{F} = \bigoplus_{i = 0}^\infty F_i t^i
$
over the Rees algebra $\widetilde\gB$ defined in the course of the proof of Theorem \ref{T:speccurve}. Then $\kF$ is the spectral sheaf of $\gB$ in the sense of Theorem \ref{T:axiomSSh}.
\end{proposition}

\begin{proof} Observe that
$
\overline{F} := \widetilde{F}/t\widetilde{F} \cong \bigoplus_{i = 0}^\infty \bigl(F_i/F_{i-1}\bigr)
$
is a torsion free module over the domain  $\overline{\gB} = \widetilde\gB/t\widetilde\gB \cong \bigoplus_{i = 0}^\infty \bigl(\gB_i/\gB_{i-1}\bigr)$.
Hence,
$\depth_{\widetilde{\gB}}\bigl(\widetilde{F}\bigr) = \depth_{\overline{\gB}}\bigl(\overline{F}\bigr) + 1 = 2,$
i.e.~$\widetilde{F}$ is a graded maximal Cohen--Macaualy module over $\widetilde{\gB}$. This implies that
$$
\Gamma(X, \kF) \cong \Hom_X(\kO, \kF) \cong \Hom_{\widetilde{\gB}}\bigl(\widetilde{\gB}, \widetilde{F}\bigr) \cong F_0 = \bigl\langle 1, \partial, \dots, \partial^{r-1}\bigr\rangle_\CC,
$$
see for example \cite[(2.2.4)]{KleimanLandofli}.
Hence, we obtain a pair $(\kF, \varphi)$ satisfying the axiomatic description of the spectral sheaf given in Theorem \ref{T:axiomSSh}.
\end{proof}

\begin{definition} The slope of a torsion free (but not necessarily locally free) coherent sheaf $\kG$ on $X$ is the ratio
$
\mu(\kG) := \frac{\chi(\kG)}{\rk(\kG)},
$
where $\chi(\kG) := \dim_{\CC}\bigl(H^0(X, \kG)\bigr) - \dim_{\CC}\bigl(H^1(X, \kG)\bigr)$ is the Euler characteristic of $\kG$ and $\rk(\kG)$ is the rank of $\kG$. A coherent sheaf  $\kG$ is semi--stable when
for any subsheaf $\kG' \subset \kG$ we have: $\mu(\kG') \le \mu(\kG)$.
\end{definition}

\begin{corollary} Let $\gB \subset \gD$  be a commutative subalgebra, $g$ be the arithmetic genus of its spectral curve $X$    and $\kF$ be its spectral sheaf.
\begin{enumerate}
\item The following sequence of coherent sheaves on $X$ is exact:
\begin{equation}\label{E:charactsequence}
0 \lar \Gamma(X, \kF) \otimes \kO \stackrel{\mathsf{ev}}\lar \kF \lar \kT \lar 0,
\end{equation}
where $\kT$ is a torsion sheaf of length $rg$ on $X$, whose support belongs to
the affine spectral curve $X_0$.
\item The sheaf $\kF$ is semi--stable of slope one.
\end{enumerate}
\end{corollary}
\begin{proof} (1) Let $\kT := \mathsf{Cok}\bigl(\Gamma(X, \kF) \otimes \kO \stackrel{\mathsf{ev}}\lar \kF \bigr)$. According to part  (3)  Theorem \ref{T:axiomSSh}, the infinite point $p \in X$ does not belong to the support of $\kT$. It implies that $\kT$ is a torsion sheaf, whose support belongs to $X_0$. Since the ranks of the torsion free sheaves $\Gamma(X, \kF) \otimes \kO$ and $\kF$ are both equal to $r$, the rank of $\ker(\mathsf{ev})$ is equal to zero. This means that $\ker(\mathsf{ev})$ is a torsion sheaf. On the other hand, $\ker(\mathsf{ev})$ is a subsheaf of a torsion free sheaf $\Gamma(X, \kF) \otimes \kO$. Therefore, $\ker(\mathsf{ev}) = 0$ and the sequence (\ref{E:charactsequence}) is exact. Taking the Euler characteristic in (\ref{E:charactsequence}) and taking into account  that $\Gamma(X, \kF) \cong \CC^r$ and $H^1(X, \kF) = 0$, we get:
$$
l(\kT) = \chi(\kT) = \chi(\kF)  - r \chi(\kO) = rg,
$$

\smallskip
\noindent
(2) Consider the following short exact sequence of coherent sheaves on X:
$$
0 \longrightarrow \kF\bigl(-[p]\bigr) \longrightarrow \kF \longrightarrow \kF\big|_{p} \longrightarrow 0.
$$
Since the evaluation map $\Gamma(X, \kF) \stackrel{\mathsf{ev}_p}\lar \kF\bigl|_{p}$ is an isomorphism
and $H^1(X, \kF) = 0$, we get the cohomology vanishing:
\begin{equation}\label{E:vanishing}
H^0\bigl(X, \kF(-[p])\bigr) = 0 = H^1\bigl(X, \kF(-[p])\bigr).
\end{equation}
We claim that the  coherent sheaf  $\widetilde\kF:= \kF(-[p])$ is semi--stable. Indeed, according to  (\ref{E:vanishing}),
$\mu(\widetilde\kF) = 0$. If $\kH$ is a subsheaf of $\widetilde\kF$ then $H^0(X, \kH) = 0$, thus $\mu(\kH) \le 0$. Hence, $\widetilde\kF$ is semi--stable, therefore  $\kF$ is semi--stable as well.
\end{proof}

\subsection{Krichever Correspondence}
\begin{definition}
Let $\gB \subset \gD$ be a commutative subalgebra. Then the triple $(X, p, \kF)$ is called \emph{spectral datum} of $\gB$. In particular, $\gB \cong \Gamma\bigl(X \setminus \{p\}, \kO\bigr)$ viewed as a $\CC$--algebra.
\end{definition}

\begin{theorem}[Krichever correspondence]\label{T:KrichCorr} Consider the following two sets:
\begin{equation}
\mathsf{DiffOp} = \left\{
\gB \subset \gD \left|
\begin{array}{l}
\gB \;\, \mbox{\rm is commutative, elliptic  and normalized}\end{array}\right.\right\}
\end{equation}
and
\begin{equation}
\mathsf{SpecData} = \left\{
(X, p, \kF) \left|
\begin{array}{l}
X \; \mbox{\rm is an integral projective curve} \\
p \in X \; \mbox{\rm is a smooth point} \\
\kF\; \mbox{\rm is torsion free},\; H^1(X, \kF) = 0 \\
\Gamma(X, \kF) \stackrel{\mathsf{ev}_p}\lar \kF\big|_p \; \mbox{\rm is an isomorphism}
\end{array}\right.\right\}.
\end{equation}
 Then the Krichever  map
\begin{equation}\label{E:KrichMap}
\mathsf{DiffOp} \stackrel{K}\lar \mathsf{SpecData}, \quad \gB \mapsto (X, p, \kF)
\end{equation}
is surjective. Moreover, its restriction $\mathsf{DiffOp}_1 \stackrel{K}\lar \mathsf{SpecData}_1$ on the set
of commutative subalgebras  $\gB \subset \gD$ of rank one, respectively the set of tuples  $(X, p, \kF)$ with $\kF$ of rank one, is essentially a bijection (the word ``essentially'' means that the spectral data of $\gB$ and $\gB'$ are the same if and only if $\gB' = \varphi(\gB)$ for  $\varphi \in \Aut(\gD)$ induced by
 $z \mapsto \alpha z$ with $\alpha \in \CC^*$).
\end{theorem}

\noindent
\emph{Comment to the proof}. In the case $X$ is a smooth Riemann surface, this result has been proven by Krichever \cite[Theorem 2.2]{Krichever77}. Singular curves and torsion free sheaves which are not locally free were included
into the picture by Mumford \cite[Section 2]{Mumford} and Verdier \cite[Proposition 4]{Verdier}. Their approach was
further developed by Mulase \cite[Theorem 5.6]{Mulase1}.

\begin{example} It was already pointed out by Burchnall and Chaundy in 1923, that the Wallenberg's family (\ref{E:Wallenberg}) exhausts the list of rank one commutative subalgebras of $\gD$, whose spectral curve $X$ is elliptic \cite[Section 8]{BurchnallChaundy}. This perfectly matches with Theorem \ref{T:KrichCorr}: in this case
$X := \CC/\Lambda \cong \Pic^0(X)$. Next, if  we wish the coefficients of the operators $P$ and $Q$ to be regular at $0$, we have to demand that the parameter $\alpha \in \CC$ from (\ref{E:Wallenberg}) does not belong to the lattice $\Lambda$. This corresponds to the exclusion of the structure sheaf $\kO$ from the
set $\Pic^0(X)$. For any $\alpha \in \CC$, consider the following function
\begin{equation}
\psi_\alpha(z, t) = \frac{\sigma(t - \alpha -z)}{\sigma(t) \sigma(z+\alpha)} \exp\bigl(\zeta(t) (z +\alpha)\bigr),
\end{equation}
where $\sigma$ and $\zeta$ are the Weiertra\ss{} elliptic functions.
Then we have:
\begin{equation}
\left\{
\begin{array}{lcl}
P_z \circ \psi_\alpha(z, t) & =  & \wp(t) \, \cdot \psi_\alpha(z, t) \\
Q_z \circ \psi_\alpha(z, t) & =  & \wp'(t) \cdot  \psi_\alpha(z, t).
\end{array}
\right.
\end{equation}
Clearly, $q = \bigl(\wp(t), \wp'(t)\bigr) \in X_0 = V(y^2 - 4x^3 + g_2x + g_3)$ for all $t \in \CC\setminus \Lambda$, where $g_2$ and $g_3$ are the Weierstra\ss{} parameters of the lattice $\Lambda$.  The function $\psi_\alpha(z, t)$ is the genus one \emph{Baker--Akhieser function}. An analogous expression for the Baker--Akhieser function exists
 for an arbitrary commutative subalgebra $\gB \subset \gD$ of rank one, such that the  spectral curve $X$ of $\gB$ is smooth, see \cite{Krichever77}. This provides    another interpretation of  the spectral sheaf $\kF$.
\end{example}

\begin{remark}
The study of commutative subalgebras of $\gD$ of arbitrary rank has been initiated by Krichever \cite{Krichever76, Krichever77, Krichever}. Although the Krichever map $K$ is surjective, the algebra $\gB$ can not be recovered from $(X, p, \kF)$  in the case $\rk(\gB) \ge 2$. In order to study this ``inverse scattering problem'', Krichever and Novikov introduced the formalism of vector--valued Baker--Akhieser functions. This method leads to explicit expressions  for commutative subalgebras of genus one and rank two \cite[Section 5]{KN} and three \cite{Mokhov}.
Using this  approach, new commutative subalgebras of rank two and higher genus  with  polynomial coefficients were recently constructed in \cite{Mironov, Mokhov2}.
\end{remark}

\begin{remark}
Commutative subalgebras $\gB \subset \gD$ with \emph{singular} spectral curve arise naturally in various applications in mathematical physics, see for instance \cite{DuistGrun, WilsonCrelle} and \cite{Taimanov}. Singular Cohen--Macaulay varieties naturally arise in Krichever' theory of partial differential operators, see \cite{KOZ}.
\end{remark}

\subsection{Fourier--Mukai transform and an approach to compute the spectral sheaf}

\medskip
\noindent
\textbf{Main question}. Assume we are given a commutative subalgebra $\gB \subset \gD$ of arbitrary rank. How to describe \emph{explicitly} its spectral sheaf $\kF$?

\smallskip
\noindent
The following observation plays a key role in our work, also explaining why the genus one case is so special.

\begin{theorem}\label{T:STtwist} The torsion sheaf $\kT$ from the short exact sequence (\ref{E:charactsequence}) is isomorphic to the \emph{Seidel--Thomas twist} of $\kF$. If the arithmetic genus of $X$ is one then the spectral sheaf $\kF$ can be recovered back from $\kT$.
\end{theorem}

\begin{proof} For any projective variety $X$ (smooth or singular) there exists an exact endo-functor $\TT = \TT_{\kO}: D^b\bigl(\Coh(X)\bigr) \lar D^b\bigl(\Coh(X)\bigr)$ of the derived category of coherent sheaves  $D^b\bigl(\Coh(X)\bigr)$ called \emph{Seidel--Thomas twist functor} \cite[Definition 2.5]{SeidelThomas},
assigning to a complex  $\kF^\bu$  another complex  $\TT(\kF^\bu)$  defined through  the distinguished triangle
\begin{equation}\label{T:twist}
\mathrm{RHom}^\bu(\kO, \kF^\bu) \stackrel{\kk}\otimes \kO \stackrel{\mathsf{ev}}\lar \kF^\bu \lar
\TT(\kF^\bu) \lar \bigl(\mathrm{RHom}^\bu(\kO, \kF^\bu)\stackrel{\kk}\otimes \kO\bigr)[1].
\end{equation}
In our case, $X$ is a curve, $\kF^\bu = \kF[0]$ is a stalk complex, $\Ext^1_X(\kO, \kF) = 0$ and the evaluation map $\Hom_X(\kO, \kF)  \otimes \kO \stackrel{\mathsf{ev}}\lar \kF$ is injective. Therefore, the distinguished triangle (\ref{T:twist}) is nothing but the short exact sequence (\ref{E:charactsequence}). The key point is the following: $\TT_\kO$ is an auto--equivalence of $D^b\bigl(\Coh(X)\bigr)$ provided $X$ is a Calabi--Yau variety \cite[Proposition 2.10]{SeidelThomas}, meaning that
\begin{equation*}
\Ext_X^i(\kO, \kO) =
\left\{
\begin{array}{cl}
\CC & i = 0, \; \dim(X) \\
0 & \mbox{otherwise}.
\end{array}
\right.
\end{equation*}
It remains to note that the irreducible Calabi--Yau curves are precisely the irreducible projective curves of arithmetic genus one (which are nothing but the Weierstra\ss{} cubics $X = X_{g_2, g_3}= \overline{V(y^2 - 4 x^3 + g_2 x + g_3)} \subset \mathbb{P}^2$,
  where $g_2, g_3 \in \CC$).
\end{proof}

\smallskip
\noindent
The above Theorem \ref{T:STtwist} implies that the torsion sheaf $\kT$ is an important invariant of the algebra $\gB$, allowing to reconstruct
 the spectral sheaf $\kF$  in the genus one case. It turns out that at least the support of $\kT$ can be algorithmically determined.

 Let $\CC\llbrace z\rrbrace$ be the field of formal Laurent series and $\widetilde\gD = \CC\llbrace z\rrbrace[\partial]$ be the algebra of ordinary differential operators with coefficients in $\CC\llbrace z\rrbrace$.
  For any character   $\gB \stackrel{\chi}\lar \CC$ consider the $\CC$--vector space
\begin{equation}\label{E:SolSpaceNew}
\mathsf{Sol}'\bigl(\gB, \chi\bigr):= \bigl\{f\in \CC\llbrace z\rrbrace \,\big|\, P\circ f = \chi(P) f \; \mbox{for all}\; P \in \gB\bigr\}.
\end{equation}
Obviously, $\mathsf{Sol}\bigl(\gB, \chi\bigr) \subseteq \mathsf{Sol}'\bigl(\gB, \chi\bigr)$. However, the following result is true.

\begin{theorem}\label{T:gcd} Let $\gB \subset \gD$ be a commutative subalgebra of rank $r$ and $\gB \stackrel{\chi}\lar \CC$ a character. Then we have: $\mathsf{Sol}\bigl(\gB, \chi\bigr) =  \mathsf{Sol}'\bigl(\gB, \chi\bigr)$ and  there exists a uniquely determined
\begin{equation}\label{E:gcd}
R_\chi = \partial^m + c_{1} \partial^{m-1} + \dots + c_m \in \widetilde\gD
\end{equation}
such that $\ker(R_\chi) = \mathsf{Sol}'\bigl(\gB, \chi\bigr)$. Moreover, $m \ge r$ and $m = r$ if and only if  $\kF$ is locally free at the point $q \in X_0$ corresponding to $\chi$. Finally,  for any $\chi$ the operator $R_\chi$ is  \emph{regular} meaning that the order of the pole of $c_i(z)$ at $z = 0$  is at most $i$ for all $1 \le i \le m$.
\end{theorem}

\begin{proof}
Let $P = \partial^n + a_{1} \partial^{n-1} + \dots + a_n \in \gD$.  Then the dimension of the $\CC$--vector space
 $\ker(P) \subset \CC\llbrace z\rrbrace$ is
 $n$ and  $\ker(P) \subset \CC\llbracket z\rrbracket$. This implies that $\mathsf{Sol}\bigl(\gB, \chi\bigr) =  \mathsf{Sol}'\bigl(\gB, \chi\bigr)$.

 \smallskip
 \noindent
 For any differential operators $Q_1, \dots,Q_l \in \widetilde\gD$ we denote by $\langle Q_1, \dots, Q_l\rangle \subseteq \widetilde\gD$ the left ideal generated by these elements.
 Recall that any left ideal  $J \subseteq \widetilde\gD$ is principal. Let $P_1, \dots, P_n \in \gB$ be the algebra generators of $\gB$ (i.e.~$\gB = \CC[P_1,\dots, P_n]$) and $\alpha_i = \chi(P_i)$ for all $1 \le i \le n$. Then there exists a uniquely determined $R_\chi \in \widetilde{\gD}$ as in (\ref{E:gcd}) such that
 \begin{equation}\label{E:DefGCD}
 \bigl\langle P - \chi(P) 1 \;\big|\; P \in \gB\bigr\rangle = \bigl\langle P_1 - \alpha_1, \dots, P_n - \alpha_n\bigr\rangle = \langle R_\chi\rangle.
 \end{equation}
 Let $\gK$ be the universal Picard--Vessiot algebra of $\CC\llbrace z\rrbrace$, see
 \cite[Section 3.2]{PutSinger}. The algebra $\widetilde\gD$ acts  on $\gK$ and any differential
 operator of order $m$ from $\widetilde\gD$ has exactly $m$ linearly independent solutions with values in $\gK$.
 Obviously, $\ker(R_\chi) = \mathsf{Sol}'\bigl(\gB, \chi\bigr) = \mathsf{Sol}\bigl(\gB, \chi\bigr)$ viewed as subspaces of $\gK$. Moreover, $\dim_{\CC}\bigl(\ker(R_\chi)\bigr) = \mathsf{ord}(R_\chi)$. In virtue of Corollary \ref{C:DimSolSpace}, we get
the statement about the order of $R_\chi$. The regularity of  $R_\chi$ follows
from a classical theorem of Fuchs, see for example  \cite[Theorem 1.1.1]{Haefliger}.
\end{proof}

\begin{definition}
In what follows, the differential operator $R_\chi$ given by (\ref{E:DefGCD})
will be called the \emph{greatest common divisor} of $P_1 - \alpha_1, \dots, P_n - \alpha_n$.
\end{definition}

\begin{theorem}\label{T:SupportSpecSheaf}  Let $\gB \subset \gD$ be a commutative subalgebra of rank $r$, $\gB \stackrel{\chi}\lar \CC$  a character, $q \in X_0$ the corresponding point and $R_\chi$ the differential operator from Theorem \ref{T:gcd}. Then $q$ belongs to the support of  $\kT$ if and only if one of the following two cases occurs.
\begin{enumerate}
\item $\mathsf{ord}(R_\chi) \ge r+1$. In this case, $q$ is a singular point of $X_0$ and the spectral sheaf $\kF$ is not locally free
at $q$.
\item $\mathsf{ord}(R_\chi) = r$ and the coefficient $c_1$ of $R_\chi$ from the expansion (\ref{E:gcd}) has a pole at $z= 0$. In this case, $\kF$ is locally free at $q$ (which is allowed to be singular).
\end{enumerate}
\end{theorem}

\begin{proof}
A point $q \in X_0$ belongs to the support of $\kT$ if and only if the evaluation map $\Gamma(X, \kF) \stackrel{\mathsf{ev}_q}\lar \kF\big|_q$ is not an isomorphism. If $\mathsf{ord}(R_\chi) \ge r+1$ then
$\mathsf{ev}_q$ is not an isomorphism from the dimension reasons. Since
$\dim_{\CC}\bigl(\kF\big|_{q}\bigr) > \rk(\kF)$, the spectral sheaf $\kF$ is not locally free at $q$. From now on assume that $\mathsf{ord}(R_\chi) = r$.
Note that the following diagram
\begin{equation}\label{E:SolSpaces}
\begin{array}{c}
\xymatrix{
\Gamma(X, \kF) \ar@{^{(}->}^-{\imath}[r] \ar[d]_{\widetilde{\eta}_\chi} \ar[dr]^{\mathsf{ev}_q} & \Gamma(X_0, \kF) \ar[d]^-{\mathsf{ev}'_q} \\
\mathsf{Sol}(\gB, \chi)^\ast & \ar[l]^-{\bar{\eta}_\chi} \kF\big|_{q}
}
\end{array}
\end{equation}
is commutative.
Recall  that $\Gamma(X_0, \kF) \cong  F = \CC[\partial]$ as $\gB = \Gamma(X_0, \kO)$--modules. The map $\imath$ is the canonical restriction map of a global section. By the construction of $\kF$, the image of $\imath$ is the linear space $\langle 1, \partial,\dots, \partial^{r-1}\rangle_\CC$, see Theorem \ref{T:axiomSSh}.  Next,
$\bar{\eta}_\chi$ is the \emph{isomorphism} (\ref{E:fiberspecsheaf}) and $\widetilde{\eta}_\chi$ assigns to the element
$\partial^i \in \Gamma(X,  \kF)$ the linear functional $\bigl(f \mapsto \frac{1}{i!}f^{(i)}(0)\bigr) \in \mathsf{Sol}(\gB, \chi)^\ast$ for all $0 \le i \le r-1$. Therefore, the map $\Gamma(X, \kF) \stackrel{\mathsf{ev}_q}\lar \kF\big|_q$ is an isomorphism if and only if
$\widetilde{\eta}_\chi$ is an isomorphism.

Now, assume that $\kF$ is locally free at the point $q$. Then the order of the differential operator $R_\chi$ is $r$, see Theorem \ref{T:gcd}. The map $\widetilde{\eta}_\chi$ is an isomorphism if and only if the solution space
$\mathsf{Sol}(\gB, \chi)$ has a basis $\bigl(z^i w_i(z) \,\big|\, 0\le i \le r-1\bigr)$ with $w_i(0) \ne 0$ for all
$0\le i \le r-1$. Since $\mathsf{Sol}(\gB, \chi) \subset \CC\llbracket z\rrbracket$, the solution space
has a basis of the form $\bigl(z^{\rho_i} \widetilde{w}_i(z) \,\big|\, 1\le i \le r\bigr)$, where $0 \le \rho_1 < \rho_2 < \dots < \rho_{r}$ and $\widetilde{w}_i(0) \ne 0$ for all $1\le i \le r$. Therefore, $\widetilde{\eta}_\chi$ is an isomorphism if and only if $(\rho_1, \dots, \rho_{r}) = (0, \dots, r-1)$. Since the singularities of the differential operator $R_\chi$ are regular, the exponents $\rho_1, \dots, \rho_{r}$ are the roots of the \emph{indicial equation}
\begin{equation}\label{E:indicialequation}
[x]_r + \gamma_1 [x]_{r-1} + \dots + \gamma_{r} = 0,
\end{equation}
where $[x]_k = x (x-1) \dots (x-k+1)$ and $\gamma_k$ is the residue of $z^{k-1} c_k(z)$ at the point $z = 0$ for all
$1 \le k \le r$, see
\cite[Section 16.11]{Ince}. Therefore, $(\rho_1, \dots, \rho_{r}) = (0, \dots, r-1)$ if and only if $\gamma_1 = 0$. This implies the statement.
\end{proof}

\smallskip
\noindent
Theorem \ref{T:SupportSpecSheaf} provides a constructive
 approach to compute the support $Z \subset X_0$ of the torsion sheaf $\kT$.
If $q \in Z$  is a smooth point of $X_0$ then the knowledge of  the roots  of the indicial equation
(\ref{E:indicialequation}) permits to extract an additional  information about the  $\kO_q$--module structure of $\kT_q$, see \cite{PrW}. To study the case when $q$ is singular,  we shall need a new ingredient: the spectral data for \emph{families} of commuting differential operators.

\subsection{On the relative spectral sheaf}

\begin{definition}
Let $R$ be an integral finitely generated $\CC$--algebra and $\gD_R = R\llbracket z\rrbracket[\partial]$. A
commutative  $R$--subalgebra $\gB \subset \gD_R$ is called \emph{elliptic} if it is flat over $R$
and there exist two monic elements $P, Q \in \gB$
(i.e.~elements whose coefficients at the highest power of $\partial$ is one)   such that
\begin{equation}\label{E:relativeelliptic}
\mathrm{gcd}\bigl(\mathsf{ord}(P), \mathsf{ord}(Q)\bigr) = \mathrm{gcd}\bigl(\mathsf{ord}(L)\; \big| \; L \in \gB\bigr).
\end{equation}
We call the number $r = \mathrm{gcd}\bigl(\mathsf{ord}(P), \mathsf{ord}(Q)\bigr)$ the rank of $\gB$.
\end{definition}

\begin{theorem}\label{T:relativeSpectData} Let $R$ be an integral finitely generated $\CC$--algebra, $B = \Spec(R)$, $\gB \subset \gD_R$ an elliptic subalgebra of rank $r$ and $X_0 = \Spec(\gB)$. Then we have:
\begin{enumerate}
\item The algebra $\gB$ is finitely generated of Krull dimension $\mathrm{kr.dim}(R) +1$.
\item There exists an algebraic variety $X_B$, flat and projective morphism
$X_B \stackrel{\pi}\lar B$ and coherent  sheaf  $\kF_B$ on $X_B$ such that
\begin{enumerate}
\item $\pi$ admits a section $B \stackrel{\sigma}\lar X_B$, whose image belongs to the regular part of $\pi$,
\item if $\Sigma = \mathsf{Im}(\sigma)$ then $X_B = \Spec(\gB)  \cup \Sigma$ and $\Spec(\gB) \cap \Sigma = \emptyset$,
\item $\kF_B$ is flat over $B$,
\item For any point $b \in B$, the tuple $\bigl(X_b, \sigma(b), \kF_b\bigr)$ is the spectral data of the algebra $R/\idm \otimes_R \gB \subset \gD$, where $\idm$ is the maximal ideal in $R$ corresponding to  $b$,
    $X_b = \pi^{-1}(b)$ and $\kF_b = \kF_B\big|_{X_b}$.
\end{enumerate}
\end{enumerate}
\end{theorem}

\begin{proof} To explain, how $X, \Sigma$ and $\kF$ are defined, we follow the exposition of \cite{KOZ}. Let $F := \gD_R/z\gD_R \cong R[\partial]$. Then $F$ is a right $\gD_R$--module with the action given by (\ref{E:action}). For any $i \in \NN_0$ we define:
$$\gB_i = \{P \in \gB \,\big|\, \mathsf{ord}(P) \le ir\} \quad \mbox{\rm and} \quad
F_i = \{Q \in F \,\big|\, \mathsf{ord}(Q) < (i+1)r\}.
$$
Consider the Rees algebra (respectively, the Rees module)
$$
\widetilde\gB: = \bigoplus\limits_{i = 0}^\infty \gB_i t^i \subset \gB[t]\quad \mbox{\rm respectively}\quad
\widetilde{F} := \bigoplus\limits_{i = 0}^\infty F_i t^i \subset F[t].
$$
Then we put $X_B := \mathsf{Proj}_R(\widetilde\gB)$ and $\kF_B := \mathsf{Proj}_R(\widetilde{F})$. The statements about $\krdim(\gB)$ and coherence of $\kF_B$ can be proven exactly in the same way as in \cite{KOZ}.

Consider the short exact sequence of $R$--modules $0 \rightarrow \gB_i \rightarrow \gB \rightarrow \gB/\gB_i \rightarrow 0$. From the assumption (\ref{E:relativeelliptic}) it follows that $\gB/\gB_i$ is a free $R$--module for
all  $i \in \NN$ sufficiently large. Since $\gB$ is flat,  $\gB_i$ is flat, too. Since $\gB_i$ is finitely generated as $R$--module, it is  projective for all $i$ sufficiently large.
The flatness of $\pi$ follows from \cite[Theorem III.9.9]{Hartshorne}. Analogously, $\kF_B$ is flat over $B$, too. Consider
$I = (t) \subset \widetilde\gB$. Then $\Sigma := V(I) \subset X_B$. See also \cite{Q}, in particular \cite[Theorem 3.15 and Lemma 4.1]{Q}, for a detailed study of the spectral data in the relative setting.
\end{proof}

\begin{remark}
In this article arise commutative subalgebras  $\gB \subset \gD_R$ with the following additional property: for any
$i \in \NN$ such that $\gB_i/\gB_{i-1} \ne 0$ there exists a monic element $L_i \in \gB_i$ with $\mathsf{ord}(L_i) = i$. In this case,   $\gB$ is free (hence flat), viewed as an $R$--module.
\end{remark}

\section{Semi--stable coherent sheaves on the Weierstra\ss{} cubic curves}\label{S:SST}
In this section, $\kk$ is an algebraically closed field of characteristic zero. We begin with a brief survey
of various techniques which were used to  study semi--stable coherent sheaves on irreducible curves of arithmetic genus one.
\subsection{Fourier--Mukai transform on the Weierstra\ss{} cubic curves}
Let $$X = X_{g_2, g_3}= \overline{V(y^2 - 4 x^3 + g_2 x + g_3)} \subset \mathbb{P}^2_\kk$$
 be a Weierstra\ss{} cubic curve, where $g_2, g_3 \in \kk$. Let $p = (0: 1: 0)$ be the infinite point of $X$
 (which is the neutral element with respect to the standard group law on the set of smooth points of $X$) and
 $X \stackrel{\imath}\lar X, (x,y) \mapsto (x, -y)$ the standard  involution of $X$.
 The following facts  are  well--known, see for example \cite{Husem}.

 \begin{theorem}\label{T:basicsGenusOne} Any  integral  projective curve of arithmetic genus one is isomorphic to an appropriate  Weierstra\ss{} cubic
 $X = X_{g_2, g_3}$.
 Moreover, if  $\delta := g_2^3 - 27g_3^2$ then  we have:
 \begin{enumerate}
 \item $X$ is smooth if and only if $\delta \ne 0$. In this case, $X$ is an elliptic curve.
 \item Assume that $\delta = 0$, i.e.~that $X$ is singular. Then $X$ has a unique singular point $s = (\xi, 0) = (\xi: 0: 1)$ with
 \begin{equation*}
 \xi =
 \left\{
 \begin{array}{ccl}
  \dfrac{3g_3}{2g_2} & g_2 \ne 0& (s \mbox{\rm \; is a nodal singularity}), \\
  0 & g_2 = 0 & (s \mbox{\rm \; is a cuspidal singularity}).
 \end{array}
 \right.
 \end{equation*}
 \end{enumerate}
 \end{theorem}

 \begin{definition} For any coherent sheaf $\kF$  on the curve $X$, we define another  coherent sheaf
$\FF(\kF):= \mathsf{Cok}\bigl(\Gamma(X, \kF) \otimes \kO \stackrel{\mathsf{ev}}\lar \kF \bigr)$,
 where $\mathsf{ev}$ is the evaluation morphism.
 \end{definition}

\begin{theorem}\label{T:FMT}
Let $X$ be a Weierstra\ss{} cubic curve, $\Sem(X)$ the category of semi--stable coherent sheaves on $X$ of slope one and $\Tor(X)$ the category of torsion coherent sheaves.  Then the  following results are true.
\begin{enumerate}
\item For any object $\kF$ of  $\Sem(X)$, the evaluation morphism $\mathsf{ev}$ is a monomorphism and the  corresponding coherent sheaf  $\FF(\kF)$ is torsion, i.e.~belongs to $\Tor(X)$.
    In other words, the  sequence (\ref{E:charactsequence}) is exact for $\kT = \FF(\kF)$.
    Moreover,
\begin{equation}
  \Sem(X) \stackrel{\FF}\lar \Tor(X)
 \end{equation}
 is  an equivalence of categories.
\item Similarly, for  any object $\kT$ of $\Tor(X)$, consider the coherent sheaf $\GG(\kT)$ given by the universal extension sequence
\begin{equation}
 0 \lar \Ext^1(\kT, \kO)^\ast \otimes \kO \lar \GG(\kT) \lar \kT \lar 0.
 \end{equation}
 Then $\GG(\kT)$ is semi--stable of slope one. Moreover, $\GG$ is  an equivalence between the categories  $\Tor(X)$ and $\Sem(X)$, which is  quasi--inverse to $\FF$.
\item For any object $\kF$ of $\Sem(X)$ and the corresponding object $\kT = \FF(\kF)$ of $\Tor(X)$ the following results are true.
    \begin{enumerate}
\item The rank of $\kF$ is equal to the length of $\kT$.
\item $\kF$ is locally free if and only if
$\kT$ has projective dimension one.
\item Analogously,  $\kF$ is not locally free if and only if the torsion sheaf
$\kT$ has infinite projective dimension. In this case,  the singular point of $X$ belongs to the support of $\kT$.
\end{enumerate}
\item Moreover, the following diagram of categories and functors is commutative:
\begin{equation}\label{E:Dualities}
\begin{array}{c}
\xymatrix{
\Sem(X) \ar[r]^-{\DD} \ar[d]_{\FF} & \Sem(X) \ar[d]^{\FF} \\
\Tor(X) \ar[r]^-{\EE} & \Tor(X),
}
\end{array}
\end{equation}
where
\begin{enumerate}
\item $\DD(\kF) := \imath^*(\kF^\vee) \otimes \kO\bigl(2[p]\bigr)$ for $\kF$ from $\Sem(X)$ with
$\kF^\vee := \mathit{Hom}_X(\kF, \kO)$.
\item $\EE(\kT) := \mathit{Hom}_X\bigl(\kT, \kK/\kO\bigr)$ is the \emph{Matlis duality} on $\Tor(X)$, see e.g.~\cite[Section 3.2]{BrunsHerzog} or \cite[Section 6]{BK1}. Here, $\kK$ is the sheaf of rational functions on $X$.
\end{enumerate}
\end{enumerate}
\end{theorem}

\noindent
\emph{Comment to the proof}. The  functorial correspondences $\FF$ and $\GG$ were  essentially introduced by  Atiyah \cite[Part II]{At}, who used them to classify indecomposable vector bundles on elliptic curves
\cite[Theorem 7]{At}. A translation of Atiyah's method into the formalism of derived categories can be for instance found in \cite{BrBur}. The idea to use the functor $\FF$ to study semi--stable sheaves on singular Weierstra\ss{} curves and elliptic fibrations (the so--called \emph{spectral cover construction}) is due to Friedman, Morgan and Witten, see  \cite[Section 1]{FMW}.
In \cite[Section 2]{BK1}, the approach  of \cite{FMW} was elaborated and included into the framework of derived categories. We refer to \cite{FMW, BK1}
 for a proof of all statements of Theorem \ref{T:FMT}, see especially   \cite[Theorem 2.21 and Theorem 6.11]{BK1}. \qed

\begin{remark}
The described equivalence between the categories $\Sem(X)$ and $\Tor(X)$ can be best understood using the Seidel--Thomas twist functor $\TT$, see Theorem \ref{T:STtwist}.
Namely,  the following diagram of categories and functors is commutative:
\begin{equation}\label{E:Functors}
\begin{array}{c}
\xymatrix{
\Sem(X) \ar[r]^-{\FF} \ar@{_{(}->}_{\mathbb{I}}[d] & \Tor(X) \ar@{^{(}->}^{\mathbb{I}}[d] \\
D^b\bigl(\Coh(X)\bigr) \ar[r]^-{\TT} & D^b\bigl(\Coh(X)\bigr),
}
\end{array}
\end{equation}
where $\mathbb{I}$ assigns to a coherent sheaf the corresponding stalk complex, see \cite[Theorem 2.21]{BK1}.
The twist functor $\TT$ is isomorphic to the \emph{integral  transform} ${\mathbbm{M}}$ with the kernel $\kP^\bu = \kI_\Delta[1]$,  where $\kI_\Delta \subset \kO_{X\times X}$ is the ideal sheaf of the diagonal
$\Delta \in X \times X$, see \cite[Lemma 3.2]{SeidelThomas}.  Recall, that the  image of an object $\kF^\bu$ from  $D^b\bigl(\Coh(X)\bigr)$ under $\MM$ is
$
{\mathbbm{M}}(\kF^\bu) := R\pi_{2\ast}\bigl(\pi_1^\ast(\kF^\bu)
\stackrel{\mathbbm{L}}\otimes
\kP^\bu\bigr),
$
where $\pi_i : X \times X \lar X$ is the canonical projection for $i = 1, 2$, see for example \cite{BBH}. In what follows, we shall call the functor $\FF$  \emph{the} Fourier--Mukai transform of $D^b\bigl(\Coh(X)\bigr)$. For a torsion free sheaf  $\kF$ from $\Sem(X)$, the corresponding torsion sheaf $\kT$ will be called \emph{Fourier--Mukai transform of $\kF$}.
\end{remark}

\begin{remark}
The formalism of integral  transforms allows to  extend  the construction of functors $\FF$ and $\GG$ to the relative setting, where
we start with a genus one fibration $X_B \stackrel{\pi}\lar  B$, see \cite{FMW, BK2, BBH}. As in the absolute case, we can define the category $\Sem(X_B/B)$ consisting of those coherent sheaves $\kF$ on $X_B$, which are flat over
$B$ and such that for any $b \in B$ the restricted sheaf $\kF\big|_{X_b}$ is semi--stable of slope one, where $X_b = \pi^{-1}(b)$. Analogously, we define the category $\Tor(X_B/B)$ of relative torsion coherent sheaves.
Again, for any object $\kF$ of $\Sem(X_B/B)$, the canonical morphism $\pi^* \bigl(\pi_*(\kF)\bigr) \lar \kF$ is a monomorphism and  we get  an equivalence of categories
$\FF_B : \Sem(X_B/B) \lar \Tor(X_B/B)$, given by the rule
$$
0 \lar \pi^*\bigl(\pi_*(\kF)\bigr) \lar \kF \lar \FF_{B}(\kF) \lar 0.
$$
Clearly, for any $b \in B$ the following diagram of categories and functors is commutative:
\begin{equation*}
\xymatrix{
\Sem(X_B/B) \ar[r]^-{\FF_B} \ar[d]_-{\imath^*_b} & \Tor(X_B/B) \ar[d]^-{\imath^*_b} \\
\Sem(X_b)  \ar[r]^-{\FF} & \Tor(X_b),
}
\end{equation*}
where $\imath_b: X_b \lar X_B$ is the inclusion of the fiber over $b$. Let $\Delta_B \subset X_B \times_B X_B$ be the relative diagonal, $\kP^\bu_B = \kI_{\Delta_B}[1]$ and $\MM_B$ the integral transform with the kernel $\kP^\bu_B$ (the relative
Fourier--Mukai transform). Then  $\MM_B$ is an auto--equivalence of the derived category $D^b\bigl(\Coh(X_B)\bigr)$ extending the equivalence $\FF_B$ similarly to the diagram (\ref{E:Functors}).
\end{remark}

\begin{theorem}\label{T:AtiyahBundle}
Let $X$ be a Weierstra\ss{} cubic curve. Then the following results are true.
\begin{enumerate}
\item For any $r \in \NN$ there exists a unique \emph{indecomposable} vector bundle $\kA_r$ of rank $r$ on $X$ recursively defined through  a  short exact sequence
\begin{equation}\label{E:AtSequence}
    0 \lar \kA_r \lar \kA_{r+1} \lar \kO \lar 0,
\end{equation}
where  $\kA_1 = \kO$.
\item Let $q \in X$ be a \emph{smooth} point, $r \in \NN$  and $\kT_{q, r} := \kO_{q}/\idm_q^r$ the indecomposable torsion sheaf of length $r$ supported at $q$. Then we have:
    $
    \GG\bigl(\kT_{q, r}\bigr) \cong \kO\bigl([q]\bigr) \otimes \kA_r.
    $
\end{enumerate}
\end{theorem}

\noindent
\textit{Comment to the proof}. The first claim was established by Atiyah \cite[Theorem 5]{At}.
The key point here is that
the category of vector bundles on $X$ admitting a filtration with quotients isomorphic to $\kO$ is
equivalent to the category of finite dimensional modules over the discrete valuation ring $\kk\llbracket t\rrbracket$.
 It follows from the definition of the functor $\FF$ that $\FF\bigl(\kO\bigl([q]\bigr) \otimes \kA_r\bigr) \cong \kT_{q, r}$, implying the second part. \qed

 \smallskip
 \noindent
 The following well--known result can be for instance found in \cite[Example 8.9 (iii)]{AK}.
\begin{proposition}\label{P:continuitybundles}
 Let $X_B \stackrel{\pi}\lar B$ be a genus one fibration with irreducible fibers admitting  a section
 $B \stackrel{\sigma}\lar X_B$ such that $\sigma(B)$ belongs to the regular part of $\pi$. Let $\kL \in
 \Pic^d(X_B/B)$, i.e.~$\kL$ is a line bundle on $X_B$ such that $\mathrm{deg}\bigl(\kL\big|_{X_b}) = d$ for all   points $b \in T$, where $X_b = \pi^{-1}(b)$.  Then there exists a unique section $B \stackrel{\jmath}\lar X_B$ with whose image
also belongs to the regular part of $\pi$ such that
 $
 \kL\big|_{X_b} \cong \kO_{X_b}\bigl((d-1)[\sigma(b)] + \jmath(b)\bigr)
 $
 for all $b \in B$.
 \end{proposition}

\subsection{Semi--stable sheaves of slope one and  rank two  on singular cubic curves} In this subsection, let
$\lambda \in \kk$ and $X = X(\lambda) := \overline{V(y^2 - x^3 - \lambda x^2)} \subset \mathbb{P}_\kk^2$  be the corresponding singular cubic curve. Let $p = (0: 1: 0)$ be the infinite point of $X$ and $s = (0: 0: 1) = (0, 0)$ its singular point. Let $R = \kk[x, y]/(y^2 - x^3 - \lambda x^2)$ be the coordinate ring of the affine  curve  $X_0 = X \setminus \{p\}$.
Let $\kA = \kA_2$ be the Atiyah bundle of rank two on $X$, see (\ref{E:AtSequence}).
For any point $t = (\alpha:\beta) \in \PP^1_\kk$, let $I_t = \bigl\langle x^2, \alpha x + \beta y\bigr\rangle \subset R$ and $T_t = R/I_t$. Finally, let $\kT_t$ be the torsion coherent sheaf on $X$ corresponding to the $R$--module $T_t$.

\begin{theorem}\label{T:classificationRankTwo}  The following results are true.
\begin{enumerate}
\item Let $\kF$ be an indecomposable semi--stable sheaf on $X$ of rank two and slope one. Then either
\begin{enumerate}
\item $\kF \cong \kA \otimes \kO\bigl([q]\bigr)$ for some \emph{smooth} point $q \in X$, or
\item $\kF \cong \kB_t := \GG(\kT_t)$ for some $t \in \PP^1_\kk$.
\end{enumerate}
\item For any  $\theta \in \kk$, let $\kB_\theta := \kB_{(\theta:\mbox{\scriptsize{\textit{1}}})}$.  Then  $\kB_\theta$ is locally free if and only if $\theta^2 - \lambda \ne 0$. In this case,
    $
\det\bigl(\kB_\theta\bigr) \cong \kO\bigl([p] + [q_\theta]\bigr),
$
where $q_\theta = \bigl((\theta^2-\lambda): \theta(\theta^2-\lambda): 1\bigr)$. In particular, $\kB_{\theta} \cong \kB_{\theta'}$ if and only if $\theta = \theta'$.
\item Similarly, $\kB_\infty := \kB_{(1: 0)}$ is locally free with $\det(\kB_\infty) \cong  \kO\bigl(2 [p]\bigr)$.
\item
    In the nodal case, the torsion free sheaves $\kU_\pm := \kB_{\pm \sqrt{\lambda}}$ are not isomorphic. In the cuspidal case, $\kU := \kB_0$ is the only indecomposable and not locally free object of $\Sem(X)$ of rank two.
\end{enumerate}
\end{theorem}

\begin{proof} (1) If $\kF$ is indecomposable then the support of its Fourier--Mukai transform $\kT := \FF(\kF)$ is a single point $q \in X$. Since $\kF$ has rank two, the length of $\kT$ is two as well. If $q \ne s$ then
$\kT \cong \kT_{q,2} = \kO_{q}/\idm_q^2$ and hence $\kF  \cong \kA \otimes \kO\bigl([q]\bigr)$.

From now on assume that $\kT$ is supported at the singular point $s$ of $X$. Let $T = \Gamma\bigl(X \setminus \{p\}, \kT)$ be the $R$--module corresponding to $\kT$. We claim that $T \cong R/J$, where $J$ is an ideal in $R$ with $\sqrt{J} = \idm_s$.
Indeed, if $\idm_s T = 0$ then $T \cong R/\idm_s \oplus R/\idm_s$ is decomposable, contradiction. Hence,
$\idm_s T \ne 0$. By Nakayama's Lemma, $\idm_s T \ne T$. Therefore, there exists elements $u \in T \setminus \idm_s T$ and $ 0 \ne v \in \idm_s T$. Since $\dim_{\kk}(T) = 2$, the elements $u$ and $v$ form a basis of $T$. Moreover,
$\langle v\rangle_\kk = \idm_s u$, i.e.~$u$ is a cyclic vector of $T$. This shows that $T \cong R/J$ for some
ideal $J \subset R$ with $\sqrt{J} = \idm_s$. But all such ideals can be classified: it can be easily shown that
$J = I_t$ for an appropriate $t \in \PP_\kk^1$.

\smallskip
\noindent (2) Let  $\kF = \kB_t$  for some $t \in \PP^1_\kk$. Then $\kF$ is locally free if any only if the $R$--module  $T = R/J$ has projective dimension one. The last  property  is true  precisely when the localized ideal
$J_s \subset R_s$ is principal.  Clearly, $I_\infty = \langle x\rangle$ is already principal in $R$. Therefore, we assume
that $t = (\theta: 1)$ for some $\theta \in \kk$  and $I_t = \langle x^2, x + \theta y\rangle$. In the ring $R$, we have the equality
$y^2 - \theta^2 x^2 = x^2 \bigl(x + (\lambda - \theta^2)\bigr)$. If $\lambda - \theta^2 \ne 0$, then $x^2$  belongs
to the ideal generated by
$y + \theta x$  in the local ring $R_s$. In particular, the localization of $I_t$ at $\idm_s$ is a principal ideal. On the other hand, if $\lambda + \theta^2 = 0$, the localization of  $I_t$   is not principal.

The determinant of the vector bundle $\kB_\theta$ can be computed using  the following trick. For any $\theta \in \kk$, consider the line
$L_\theta$ given by the equation $y + \theta x = 0$. Since  $\lambda - \theta^2 \ne 0$, the line  $L_\theta$ intersects
the curve $X$ at two points: the singular point $s$ and another point $q_\theta'= \bigl(\theta^2-\lambda: -\theta(\theta^2-\lambda): 1\bigr)$. Moreover, we have:
$
\kC_\theta := \GG(R/L_\theta) \cong \kB_\theta \oplus \kO\bigl([q_\theta']\bigr).
$
 Now we claim that $\det\bigl(\kC_\theta\bigr) \cong \kO\bigl(3[p]\bigr)$. Indeed, consider the constant genus one  fibration $X_B = X \times  B$ over the base $B = \Spec\bigl(\kk[\tau]\bigr)$ and the $B$--flat family of
torsion sheaves given by the $\kk[x,y,\tau]/(y^2 - x^3 - \lambda x^2)$--module
$\widetilde{L}: = \kk[x,y,\tau]/(y^2 - x^3 - \lambda x^2, y - \tau + \theta x)$.
Using the inverse relative Fourier--Mukai transform $\GG_B$, we get a family $\widetilde\kC := \GG_B(\widetilde{L})$ of relatively semi--stable vector bundles on $X \times  B$ with
$\widetilde{\kC}\big|_{X \times \{0\}} \cong \kC_{\theta}$. For $\zeta \ne 0$,  we have:
$
\widetilde{\kC}\big|_{X \times \{\zeta\}} \cong \kO([p_1]) \oplus \kO([p_2]) \oplus \kO([p_3]),
$
where $p_1, p_2$ and $p_3$ are the intersection points of $X$ with  the line $V(\varphi_{\theta, \zeta})$, where
 $\varphi_{\theta, \zeta}(x, y):= y - \zeta + \theta x$. However, the divisor of the function $\varphi_{\theta, \zeta}$ is $[p_1] + [p_2] + [p_3] - 3 [p]$.
It means that
$
\det\bigl(\widetilde\kC\big|_{X \times \{\zeta\}}\bigr) \cong \kO\bigl(3[p]\bigr)
$
for all $\zeta\ne 0$. From Proposition \ref{P:continuitybundles} easily follows that $\det(\kC_\theta) \cong \det\bigl(\widetilde\kC\big|_{X \times \{0\}}\bigr) \cong \kO\bigl(3[p]\bigr)$
as well.  This fact implies that
$
\det(\kB_\theta) \cong \kO\bigl(3 [p]) \otimes \kO\bigl([q_\theta']\bigr)^\vee \cong \kO\bigl(3 [p]) \otimes \kO\bigl([q_\theta] - 2 [p]\bigr) \cong \kO\bigl([p]+[q_\theta]\bigr).
$

\smallskip
\noindent (3) Note that for $\theta \ne 0$ we have:
$q_\theta=
\bigl((\theta^2 - \lambda): \theta (\theta^2 - \lambda) : 1\bigr) =
\bigl((\xi - \lambda \xi^{3}) : (1 - \lambda \xi^2) : \xi^3\bigr)$
and
$
I_t = \langle x^2, x + \xi y\rangle,
$
where $\xi = \theta^{-1}$.  From the continuity consideration similar to the one given in the previous paragraph, we deduce   that $\det(\kB_\infty) \cong \kO\bigl(2[p]\bigr)$.

\smallskip
\noindent (4) Let $\lambda \ne 0$, i.e.~the curve $X$ is nodal. We \emph{choose} a square root $\rho = \sqrt{\lambda}$ and consider the ideal $I = \langle x^2, y + \rho x\rangle$ in the local ring $R_s$. Let
$\widehat{R}$ be the completion of $R_s$ and $\idm$ the maximal ideal of $\widehat{R}$. Consider the element
$ \widehat{R} \ni
w = x \sqrt{\lambda + x} := \rho x + \frac{1}{2\rho} x^2 + \dots
$
Writing $x$ as a power series in $w$ we conclude that $x \equiv \frac{1}{\rho} w \; \mod \; \idm^2$. Posing
$
u_\pm = y \pm x \sqrt{\lambda + x} := y \pm w \in \widehat{R}
$
we get: $\widehat{R} = \kk\llbracket u_+, u_-\rrbracket/(u_+ u_-)$. The next step is to determine the images
of the generators of $I$ under the completion map $R_s \rightarrow \widehat{R}$.
We see that $y + \rho x \equiv u_+ \; \mod \; \idm^2$ and $x^2 = \frac{1}{4\rho^2}\bigl(u_+^2 + u_-^2\bigr) \;  \mod \; \idm^4$. Therefore, we conclude that
\begin{equation}\label{E:Upm}
U_+:= R/(x^2, y + \rho x) \cong \widehat{R}/(u_+, u_-^2) \quad \mbox{and} \quad U_-:= \widehat{R}/(x^2, y -  \rho x) \cong \widehat{R}/(u_-, u_+^2).
\end{equation}
In particular,
$\kU_+ \not\cong \kU_-$, where $\kU_\pm$ are torsion sheaves on $X$ corresponding to $U_\pm$.
\end{proof}

\begin{remark} Let $X$ be a singular Weierstra\ss{} cubic with the singular point $s$ and the ``infinite'' point $p$.
According to Theorem \ref{T:classificationRankTwo},  for any smooth point $q \in X$ there exists a unique semi--stable vector
bundle with determinant $\kO\bigl([p] + [q]\bigr)$, whose  Fourier--Mukai transform  is supported at the singular point $s$.
Abusing the notation, we shall denote this vector bundle by $\kB_q$ in what follows. Such description of vector bundles from  $\Sem(X)$ is advantageous since it eliminates unessential choices (for example, the dependence of $\lambda$ in the nodal case, see part (2) of Theorem \ref{T:classificationRankTwo}).
\end{remark}

\begin{corollary}\label{C:ListRankTwo} Let $X$ be a singular Weierstra\ss{} cubic curve, $p \in X$ its point at infinity, $\PP^1 \stackrel{\nu}\lar X$ the  normalization morphism,
$\kS = \nu_* \bigl(\kO_{\PP^1}\bigr)$ and $\kA = \kA_2$  the rank two Atiyah bundle on $X$. Let $\kF$ be a semi--stable torsion free sheaf  on $X$ of rank two and slope one, $\kT$ be the Fourier--Mukai transform of $\kF$ and $Z = \mathsf{Supp}(\kT)$.
\begin{enumerate}
\item If $\kF$ is locally free and indecomposable, then it is either isomorphic to $\kA \otimes \kO\bigl([q]\bigr)$ for
some smooth point $q \in X$ or to $\kB_{\bar{q}}$, where $\det\bigl(\kB_{\bar{q}}\bigr) = \kO\bigl([\bar{q}] + [p]\bigr) \in \mathsf{Pic}^2(X)$, where $\bar{q}$ is a smooth point of $X$ (which  can be arbitrary). In the first case  $Z = \{q\}$, whereas in the second  $Z = \{s\}$.
\item If $\kF$ is indecomposable but not locally free, then it is isomorphic to one of the sheaves $\kU_\pm$ (nodal case) or to $\kU$ (cuspidal case). In this case, $Z = \{s\}$.
\item If $\kF$ is decomposable, then it is isomorphic to $\kO\bigl([q]\bigr) \oplus \kO\bigl([q']\bigr)$,
$\kO\bigl([q]\bigr) \oplus \kS$ or $\kS \oplus \kS$ for some smooth points $q, q' \in X$. We have: $Z = \{q, q'\},
\{q, s\}$ or $\{s\}$ respectively.
\end{enumerate}
 For any object $\kF$ of $\Sem(X)$ we have: $H^1(X, \kF) = 0$. Moreover,
$\Gamma(X, \kF) \stackrel{\mathsf{ev}_p}\lar \kF\big|_p$ is an isomorphism if and only if $p \notin Z$.
\end{corollary}

\begin{remark}\label{R:HilbertSchemes} In fact, one can derive  from \cite{Ran} the following result. For   $\lambda \in \kk$, let
$X = \overline{V(y^2 - x^3 - \lambda x^2)}$  and  $s = (0, 0)$ be the singular point of $X$. Let
$\mathsf{Hilb}^2_s(X)$ be the Hilbert scheme of points of length two on $X$, supported at $s$. Then $\mathsf{Hilb}^2_s(X) \cong \mathbb{P}^1$. Moreover, the corresponding \emph{universal ideal} $\kJ \subset \kO_{X \times \PP^1}$ is $\langle x^2, z_0 x - z_1 y\rangle$, where $(z_0: z_1)$ are homogeneous coordinates on $\PP^1$.
\end{remark}

\begin{remark}\label{R:GeomDescrSS} Indecomposable torsion free sheaves $\kB_q$ and $\kU_\pm$ on the nodal cubic curve $X = \overline{V(y^2 - x^3 -  x^2)}$ admit the following explicit  description, see \cite[Theorem 5.1]{BK1}. Let $Y$ be a Kodaira cycle of two projective lines, $I$ a chain of two projective lines, $Y \stackrel{\pi}\lar X$ an \'etale covering of degree two and $C \stackrel{\kappa}\lar X$ a finite  morphism of degree two (which is the  composition of $\pi$ with a partial normalization map $C \lar Y$).
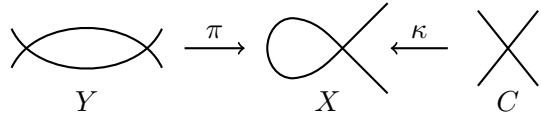
\begin{center}
\begin{tikzpicture}[scale=0.2]
    \draw [thick] (3,3)
    to [out=225,in=45] (0,0)
      to [out=225,in=5] (-3.3,-2)
	to [out= 180, in=-90](-5,0)
	to [out= 90, in=180](-3.3,2)
      to [out= -5, in=135](0,0)
    to [out= 315, in=135] (3,-3)  ;

    \node at(-1,-3.5){$X$};

   \draw[thick,->](-10.5,0) to node[above]{$\pi$} (-6.5,0);
    \draw[thick] (-22,-1.3) to [bend left=65] (-12,-1.3);
    \draw[thick] (-12,1.3) to [bend left=65] (-22,1.3);

       \node at(-17,-3.5){$Y$};
   \draw[thick,<-](3.2,0) to node[above]{$\kappa$}(7,0);
 \draw[thick] (9,-2.5) to  (13,2.5);
 \draw[thick] (9, 2.5) to  (13,-2.5);

 \node at(11,-3.5){$C$};

\end{tikzpicture}
\end{center}
It is not difficult to show that the map  $\Pic(C) \stackrel{\underline{\deg}}\lar \ZZ^2$, assigning to a line bundle on $C$ the degrees of its restrictions on every irreducible component of $C$, is an isomorphism of abelian groups.  Similarly, there is an isomorphism of abelian groups  $\Pic(Y) \xrightarrow{(\underline{\deg},   \gamma)} \ZZ^2 \times \kk^\ast$ (however, the  component $\gamma$  of this map is not canonical).
\begin{itemize}
\item Consider the line bundles $\kL_+ = \kO(1, 0)$ and $\kL_- = \kO(0, 1)$ on $C$. Taking  appropriate choices (see part (4) of the proof of Theorem  \ref{T:classificationRankTwo}) we have:  $\kU_\pm \cong \kappa_* \kL_\pm$.
\item Similarly, consider the  line bundle $\kL\bigl((2,0), \lambda\bigr)$ with $\lambda \in \kk^*$. Then we have:
$\pi_\ast \bigl(\kL\bigl((2,0), \lambda\bigr)\bigr) \cong \kB_q$ for some smooth point $q \in X$.
\end{itemize}
\end{remark}

\subsection{Regular semi--stable sheaves on a cuspidal cubic curve} In this subsection, $X = \overline{V(y^2 - x^3)} \subset \PP_\kk^2$ is a cuspidal cubic curve, $R = \kk[t^2, t^3] = \kk[x,y]/(y^2 - x^3)$ and
$
\widehat{R} := \kk\llbracket t^2, t^3\rrbracket
$
is the completed local ring of $X$ at the singular point $s = (0, 0)$. According to a result of Drozd \cite{Drozd}, the category of finite dimensional $\widehat{R}$--modules is \emph{representation wild}. This means that
for any finitely generated $\kk$--algebra $\Lambda$ there exists an exact functor
$
\Lambda-\mathsf{fdmod} \stackrel{\JJ}\lar \widehat{R}-\mathsf{fdmod}
$
such that
\begin{itemize}
\item $\JJ(M) \cong \JJ(M')$ if and only if $M \cong M'$.
\item $\JJ(M)$ is indecomposable if and only if $M$ is indecomposable.
\end{itemize}
See also \cite[Proposition 8]{Survey} for a more detailed discussion of representation wildness and a simpler proof of Drozd's result. Therefore, the category
$\Sem(X)$ is representation--wild, too. Nevertheless, in this subsection we shall give a full classification
of the indecomposable objects of $\Sem(X)$ having  rank three.

\begin{definition}
For any $n \in \NN_0$ and $\theta \in \kk$ consider the following ideals in $\widehat{R}$:
\begin{equation}
I_{n,\theta} = \bigl\langle t^n (t^2 + \theta t^3)\bigr\rangle \quad \mbox{and}\quad
J_n = t^n \langle t^2, t^3\rangle.
\end{equation}
\end{definition}

\begin{lemma}\label{L:ideals}
Let $I \subset \widehat{R}$ be a proper ideal. Then we have: $I = I_{n,\theta}$ or $I = J_n$ for some $n \in \NN_0$ and $\theta \in \kk$.
\end{lemma}

\begin{proof}
Let $f = t^m + \theta t^{m+1} + \dots  =
t^m(1 + \theta t + \dots) = t^m \cdot w \in \kk\llbracket t\rrbracket$ be an element of $I$ with the minimal  multiplicity $m \in \NN_{\ge 2}$, where $\theta \in \kk$ is some scalar. Then for any $k \in \NN_{\ge 2}$ the power series  $t^k w^{-1}$  belongs to $\widehat{R}$. Therefore, $t^{k+m}$ belongs to the principal
ideal $(f)$ in $\widehat{R}$, provided $k\ge 2$. Now, the following two cases can occur.

\smallskip
\noindent
\underline{Case 1}. The ideal $I$ contains an element of multiplicity  $m+1$. Then
$
I = \langle t^m, t^{m+1}\rangle.
$

\smallskip
\noindent
\underline{Case 2}. The ideal $I$ does not contain any elements of multiplicity  $m+1$. Then  $I = \langle f\rangle$. \end{proof}

\begin{theorem}\label{T:listRk3}
Let $M$ be an indecomposable $\widehat{R}$--module with $\dim_{\kk}(M) = 3$. Then $M$ is isomorphic to some   module from the following list:
\begin{enumerate}
\item $M_\theta:= \widehat{R}/(t^3 + \theta t^4)$, where $\theta \in \kk$.
\item $N := \widehat{R}/(t^4, t^5)$.
\item $N^\sharp := \EE(N)$ (the Matlis dual of $N$).
\end{enumerate}
Moreover, $\mathsf{pr.dim}_{\widehat{R}}(M_\theta) = 1$ and $\mathsf{pr.dim}_{\widehat{R}}(N) = \mathsf{pr.dim}_{\widehat{R}}(N^\sharp) = \infty$.
\end{theorem}

\begin{proof}
For any $n \in \NN$,  an $n$--dimensional $\widehat{R}$--module  is  determined by an algebra homomorphism
$\widehat{R} \longrightarrow  \Mat_{n \times n}(\kk)$. Since $\widehat{R} \cong \kk\llbracket u, v\rrbracket/(v^2 - u^3)$,  such a homomorphism  is specified by a pair of \emph{nilpotent} matrices
$U, V \in \Mat_{n \times n}(\kk)$ satisfying the conditions
$$
UV = VU \quad \mbox{and} \quad V^2 = U^3.
$$
Moreover, two such pairs $(U, V)$ and $(U', V')$ define isomorphic $\widehat{R}$--modules  if and only if there exists
a matrix $S \in \GL_n(\kk)$ satisfying
$$
U' = S U S^{-1} \quad \mbox{and} \quad V' = S V S^{-1}.
$$

\smallskip
\noindent
\underline{Case 1}. Assume that $\rk(U) = 2$.  Then we may without loss of generality assume that
$U = \left(
\begin{array}{ccc}
0 & 1 & 0 \\
0 & 0 & 1 \\
0 & 0 & 0
\end{array}
\right).
$ The equalities $UV = VU$ and $V^2 = 0$ imply that $V = \left(
\begin{array}{ccc}
0 & 0 & \theta \\
0 & 0 & 0 \\
0 & 0 & 0
\end{array}
\right)
$
for some $\theta \in \kk$. It is easy to see that
$
\bigl(\kk^3, U, V\bigr) \cong M_{-\theta}.
$

\smallskip
\noindent
\underline{Case 2}. Assume that $\rk(U) = 1$. Then we may without loss of generality assume that
$U = \left(
\begin{array}{ccc}
0 & 0 & 1 \\
0 & 0 & 0 \\
0 & 0 & 0
\end{array}
\right).
$
From the equalities $UV = VU$ and $V^2 = 0$ we conclude that $V = \left(
\begin{array}{ccc}
0 & \alpha & \gamma \\
0 & 0 & \beta \\
0 & 0 & 0
\end{array}
\right)
$
for some $ \alpha, \beta, \gamma \in \kk$ such that $\alpha \beta = 0$.
In the case $\alpha = 0 = \beta$, the module $M$ contains  the trivial module $\kk = \widehat{R}/(t^2, t^3)$ as a direct summand. In particular, $M$ is decomposable.

\smallskip
\noindent
Assuming that $\alpha = 0$ and $\beta \ne 0$, we see that
$$
\left(
\kk^3,
\left(
\begin{array}{ccc}
0 & 0 & 1 \\
0 & 0 & 0 \\
0 & 0 & 0
\end{array}
\right),
\left(
\begin{array}{ccc}
0 & 0 & \gamma \\
0 & 0 & \beta \\
0 & 0 & 0
\end{array}
\right)
\right) \cong \left(
\kk^3,
\left(
\begin{array}{ccc}
0 & 0 & 1 \\
0 & 0 & 0 \\
0 & 0 & 0
\end{array}
\right),
\left(
\begin{array}{ccc}
0 & 0 & 0 \\
0 & 0 & 1 \\
0 & 0 & 0
\end{array}
\right)
\right) \cong N.
$$
Similarly, if $\beta = 0$ and $\alpha \ne 0$, we have:
$$
\left(
\kk^3,
\left(
\begin{array}{ccc}
0 & 0 & 1 \\
0 & 0 & 0 \\
0 & 0 & 0
\end{array}
\right),
\left(
\begin{array}{ccc}
0 & \alpha & \gamma \\
0 & 0 & 0 \\
0 & 0 & 0
\end{array}
\right)
\right) \cong \left(
\kk^3,
\left(
\begin{array}{ccc}
0 & 0 & 1 \\
0 & 0 & 0 \\
0 & 0 & 0
\end{array}
\right),
\left(
\begin{array}{ccc}
0 & 1 & 0 \\
0 & 0 & 0 \\
0 & 0 & 0
\end{array}
\right)
\right) \cong N^\sharp.
$$
The last isomorphism follows from the well--known fact that the Matlis duality in the category of finite dimensional $\widehat{R}$--modules is given by the rule
$(U, V) \mapsto (U^t, V^t)$, where $U^t$ is   the transposed matrix of $U$,  see for example \cite[Remark 6.5]{BK1}.

\smallskip
\noindent
\underline{Case 3}. Finally, for  $U = 0$ it is easy to show that $M$ contains the trivial module $(\kk, 0, 0)$
as a direct summand.

\smallskip
\noindent
The statement about the projective dimension of $M$ is obvious.
\end{proof}

\begin{corollary}\label{C:ListRankThree} An indecomposable semi--stable coherent sheaf of rank three and slope one
on a cuspidal cubic curve $X$ is isomorphic to a one of the following sheaves:
\begin{enumerate}
\item $\kO\bigl([q]\bigr) \otimes \kA_3$, where $\kA_3$ is the Atiyah bundle of rank three from Theorem \ref{T:AtiyahBundle} and $q \in X$ is a smooth point.
\item $\kE_{q(\theta)} := \GG(M_\theta)$ for some $\theta \in \kk$, where $M_\theta$ is the $\widehat{R}$--module
from Theorem \ref{T:listRk3}, viewed as a torsion sheaf on $X$.  Moreover, $\kE_{q(\theta)}$ is locally free and $\det(\kE_{q(\theta)}) \cong \kO\bigl([q_\theta] + 2[p]\bigr),$ where $q_{\theta} = (\theta: 1: \theta^3)$, see the next Lemma \ref{L:CuspRegular}.
\item $\kV := \GG(N)$ and $\kV^\dagger := \GG(N^\sharp)$. They are not locally free and $\DD(\kV) \cong \kV^\dagger$, where $\DD$ is the duality on $\Sem(X)$ from
Theorem \ref{T:FMT}.
\end{enumerate}
\end{corollary}

\smallskip
\noindent
The following class of indecomposable semi--stable vector bundles on a cuspidal cubic curve $X$ was  introduced by Friedman, Morgan and Witten \cite{FMW}.

\begin{lemma}\label{L:CuspRegular}
For any $n \in \NN_{\ge 2}$ and $\theta \in \kk$, consider the $\widehat{R}$--module $T_{n, \theta} = \widehat{R}/(t^n  + \theta t^{n+1})$,  which we view as a torsion sheaf on $X$.  Let $\kE(n, \theta) := \GG\bigl(T(n, \theta)\bigr)$. Then we have:
\begin{enumerate}
\item $\kE(n, \theta)$ is an indecomposable locally free sheaf  of rank $n$ on $X$ with
$$\det\bigl(\kE(n, \theta)\bigr) \cong \kO\bigl([q_\theta] + (n-1)[p]\bigr), \;  \mbox{\rm where}\;\, q_{\theta} = (\theta: 1: \theta^3).$$
\item $\DD\bigl(\kE(n, \theta)\bigr) \cong \kE(n, \theta)$, where $\DD$ is the duality on $\Sem(X)$ from
Theorem \ref{T:FMT}.
\end{enumerate}
\end{lemma}

\begin{proof} The fact that $\kE(n, \theta)$ is an indecomposable locally free sheaf  of rank $n$ follows from the fact that $T_{n, \theta}$ is an indecomposable  $R$--module with $\dim_{\kk}(T_{n, \theta}) = n$ and
$\mathsf{pr.dim}_{\widehat{R}}(T_{n, \theta}) = 1$, combined with  Theorem \ref{T:FMT}.
 Consider the vector bundle $\kC(n, \theta) := \GG\bigl(R/(t^n  + \theta t^{n+1})\bigr)$. As in the proof of Theorem \ref{T:classificationRankTwo} we show that
\begin{itemize}
\item $\kC(n, \theta) \cong \kE(n, \theta)\oplus \kO\bigl([q_\theta']\bigr)$, where $q_{\theta}' = (\theta: -1: \theta^3)$.
\item $\det\bigl(\kC(n, \theta)\bigr) \cong \kO\bigl((n+1)[p]\bigr)$.
\end{itemize}
This implies that
$
\det\bigl(\kE(n, \theta)\bigr) \cong \kO\bigl((n+1)[p] - [q_\theta']\bigr) \cong \kO\bigl([q_\theta] + (n-1)[p]\bigr).
$

Next, note that $T_{n, \theta}$ has a one--dimensional socle generated by the class of $t^n$. Therefore, its Matlis dual module $T_{n, \theta}^\sharp := \EE(T_{n, \theta})$ has a simple top. Since $\dim_{\kk}(T_{n, \theta}^\sharp) = \dim_{\kk}(T_{n, \theta}) = n$, Lemma \ref{L:ideals} implies existence of  some $\tilde\theta \in \kk$ with $T_{n, \theta}^\sharp \cong T_{n, \tilde\theta}$. Therefore, $\DD\bigl(\kE(n, \theta)\bigr) \cong \kE(n, \tilde\theta)$. On the other hand, it is easy to see that $\det\bigl(\DD(\kF)\bigr) \cong \det(\kF)$ for any
locally free object of $\Sem(X)$. Therefore, $\tilde{\theta} = \theta$.
\end{proof}

\begin{remark}
A full classification of all indecomposable semi--stable coherent sheaves of arbitrary integral slope on a nodal Weierstra\ss{} curve, similar to the one given in Remark \ref{R:GeomDescrSS},  was given in \cite[Theorem 5.1]{BK1}.
\end{remark}

\section{Spectral sheaves  of rank two and genus one commutative subalgebras}\label{S:SpecSheaves}

\noindent
In this section, we classify the spectral sheaves of all rank two and genus one commutative subalgebras of $\gD$
with singular spectral curve, completing the result of Previato and Wilson \cite[Theorem 1.2]{PrW}.

\subsection{Gr\"unbaum's classification} We begin by  recalling  the  classification  of  rank two and genus one commutative subalgebras of $\gD$, following  Gr\"unbaum's work \cite{Grun}. In  the next, $\mathfrak{E} = \CC\llbracket z\rrbracket\llbrace \partial^{-1}\rrbrace$ is the algebra of pseudo--differential operators and for any $Q \in \mathfrak{E}$ we denote by
$Q_+$ the ``differential part'' of $Q$, i.e.~the projection of $Q$ onto $\gD$. We refer to \cite{Mulase1} and
\cite[Appendix A]{Q} for a survey of properties of the algebra $\gE$.

\smallskip
\noindent
The following result can be found in
\cite[Section 2]{Grun}, see also  \cite[Lemma 5.2]{PrW}. For the reader's convenience, we give a detailed proof here.

\begin{proposition}\label{P:RankTwo}
Let $\gB \in \gD$ be a normalized commutative subalgebra of rank two and genus one. Then there exist two operators $L, M \in \gB$ such that   $\gB = \CC[L, M]$ and
\begin{equation}\label{E:Ansatz}
L = \partial^4 + a_2 \partial^2 + a_1 \partial + a_0,  \quad M = 2  L^{\frac{3}{2}}_+,\quad M^2 = 4 L^3 - g_2 L - g_3
\end{equation}
for some $g_2, g_3 \in \CC$.
\end{proposition}

\begin{proof} If  $X$ is the spectral curve of $\gB$ then
$\gB \cong \Gamma\bigl(X\setminus \{p\}, \kO\bigr)$ as associative algebras. As $\gB$ has genus one, there exist $g_2, g_3 \in \CC$ such that $\Gamma\bigl(X\setminus \{p\}, \kO\bigr) \cong \CC[x, y]/(y^2 - 4x^3 + g_2x + g_3)$. In particular, we can find a pair of operators $L, M \in \gB$ such that $\gB = \CC[L, M]$ and
$M^2 = 4 L^3 - g_2 L - g_3$. As the rank of $\gB$ is two, $\mathsf{ord}(L) = 4$ and $\mathsf{ord}(M) = 6$ is the only possibility. Since $\gB$ is normalized, the operator $L$ is normalized. Our next goal is to show that we can find a change of variables
$$
\left\{
\begin{array}{cccl}
L & \mapsto & \widehat{L} & = L + \alpha \\
M & \mapsto & \widehat{M} & = M + \beta + \gamma L
\end{array}
\right.
$$
with $\alpha, \beta, \gamma \in \CC$ such that $\widehat{M} = 2  \widehat{L}^{\frac{3}{2}}_+$.
From the theory of pseudo--differential operators we know that there exists a uniquely determined operator
$
L^{\frac{1}{4}} = \partial + b_{1} \partial^{-1} + b_{2} \partial^{-2} + \dots$ in  $\gE$.
Then for any $i \in \ZZ$ we have:
$
L^{\frac{i}{4}} = \partial^{i} + b_{1}^{(i)} \partial^{i-2} + b_{2}^{(i)} \partial^{i-3} + \dots
$
If $M \in \gD$ is such that $[L, M] = 0$ and $\mathsf{ord}(M) = 6$ then there exist constants $\gamma_i \in \CC$ for
$i \in \ZZ_{\le 6}$ such that
\begin{equation}\label{E:PseudodiffExpansion}
M = \sum\limits_{i = 6}^{- \infty} \gamma_i L^{\frac{i}{4}} = \sum\limits_{i = 6}^{0} \gamma_i L^{\frac{i}{4}}_{+}.
\end{equation}
Rescaling, assume that $\gamma_6 = 1$. Next, we have the following identity in the algebra $\gB \subset \gE$:
$$
M^2 - L^3 = \bigl(2 \gamma_5 L^{\frac{11}{4}} + \dots\bigr) = \bigl(2 \gamma_5 L^{\frac{11}{4}} + \dots\bigr)_+ = 2 \gamma_5 \partial^{11} + \mathrm{l.o.t.}
$$
Since $\gB$ has  rank two, it does not contain any differential operators of odd order. Therefore, $\gamma_5 = 0$ and $$
M = L^{\frac{3}{2}} + \gamma_4 L + \gamma_3 L^{\frac{3}{4}} + \gamma_2 L^{\frac{1}{2}} + \dots =
L^{\frac{3}{2}}_+ + \gamma_4 L + \gamma_3 L^{\frac{3}{4}}_+  + \gamma_2 L^{\frac{1}{2}}_+ + \gamma_1 L^{\frac{1}{4}}_+ +
\gamma_0.
$$
Consider $N := M - \gamma_4 L \in \gB$. Again, we get the following equality in $\gB$:
$$
N^2 - L^3 = \bigl(2 \gamma_3 L^{\frac{9}{4}} + \dots\bigr) = \bigl(2 \gamma_3 L^{\frac{9}{4}} + \dots\bigr)_+ = 2 \gamma_3 \partial^{9} + \mathrm{l.o.t.}
$$
This implies that $\gamma_3 = 0$, too.

Let $\alpha \in \CC$ and $\widehat{L} = L + \alpha$. Obviously, we have: $\CC[L, M] = \CC[\widehat{L}, M]$. Moreover, $\widehat{L}^{\frac{3}{2}}_+ = {L}^{\frac{3}{2}}_+  + \frac{3}{2} \alpha {L}^{\frac{1}{2}}_+$ and  $\widehat{L}^{\frac{i}{4}}_+ = {L}^{\frac{i}{4}}_+$ for $i = 1, 2$. Therefore, we get a yet new identity in  $\gB$:
$$
\widehat{M} := M - \gamma_4 L - \gamma_0 = \widehat{L}^{\frac{3}{2}} + \gamma_1 \widehat{L}^{\frac{1}{4}} +
\gamma_{-1} \widehat{L}^{-\frac{1}{4}} + \dots =
\widehat{L}^{\frac{3}{2}}_+ + \gamma_1 \widehat{L}^{\frac{1}{4}}_+.
$$
As in the previous steps, we get an element
$
\widehat{M}^2 - \widehat{L} =  2 \gamma_1 \partial^{7} + \mathrm{l.o.t.} \in \gB
$
implying that $\gamma_1 = 0$, as $\rk(\gB) = 2$.
\end{proof}

\smallskip
\noindent
The following result is due to Gr\"unbaum \cite{Grun}.
\begin{theorem}\label{T:Gruenbaum} Let $\gB \subset \gD$ be a genus one and rank two commutative subalgebra. Then
$$\gB = \CC[L, M] = \CC[x, y]/(y^2 - 4x^3 + g_2x + g_3)$$
for some parameters $g_2, g_3 \in \CC$. Here,
\begin{equation}\label{E:operatorL}
L = \Bigl(\partial^2 + \frac{1}{2} c_2\Bigr)^2 + \bigl(c_1 \partial + \partial c_1) + c_0
\end{equation}
for certain $c_0, c_1, c_2 \in \CC\llbracket z\rrbracket$ obeying further constraints described below and $M = 2 L^{\frac{3}{2}}_+$.

\smallskip
\noindent
1.~In the so--called \emph{formally self--adjoint case}, $c_1 = 0$ and the following two subcases occur:
\begin{enumerate}
\item $c_0$ is a constant. Then the spectral curve is
$
y^2 = 4 x^3  - 3 c_0^2 x - c_0^3.
$
\item $c'_0 \ne 0$. Then $c_0 = f$  and $c_2$ is given by the formula
\begin{equation}\label{E:coefc2}
c_2 = \frac{K_2 + 2K_3 f + f^3 - f''' f' + \frac{1}{2} (f'')^2}{f'^2}.
\end{equation}
for some  $K_2, K_3 \in \CC$. Other way around, if $f, K_2, K_3$ are such that $c_2$ is regular at $z = 0$ then $\gB = \CC[L, M]$ has genus one and rank two.
The spectral curve of $\gB$  is given by the equation
$
y^2 = 4 x^3  + 2 K_3  x - \dfrac{K_2}{2}.
$
\end{enumerate}

\smallskip
\noindent
2.~In the ``generic''  non--self--adjoint case, $c_0, c_1$ and $c_2$ are given by the formulae
\begin{equation}\label{E:GruenbaumCoeff}
    \left\{
    \begin{array}{ccl}
    c_0 & = & -f^2 + K_{11} f + K_{12} \\
    c_1 & = & f' \\
    c_2 & = & \dfrac{K_{14} - 2 K_{10} f + 6 K_{12} f^2 + 2 K_{11} f^3  - f^4 + f''^2 - 2 f' f'''}{2 f'^2}
    \end{array}
    \right.
\end{equation}
where $f \in \CC\llbracket z\rrbracket$ satisfies $f(0) = 0$, and $K_{10}, K_{11}, K_{12}, K_{14} \in \CC$. Other way around, if $f, K_{10}, K_{11}, K_{12}, K_{14}$ are such that $c_2$ is regular at $z = 0$ then $\gB = \CC[L, M]$ has genus one and rank two.
In this case, the Weierstra\ss{} parameters $g_2$ and $g_3$ of the spectral curve are given by the expressions
$$g_2  =  3 K_{12}^2  +  K_{10} K_{11} - K_{14}\; \mbox{\rm and}\;
g_3 = \frac{1}{4}\bigl(2 K_{10} K_{11} K_{12} + 4 K_{12}^3 + K_{14}(K_{11}^2 + 4 K_{12}) - K_{10}^2)\bigr).$$
\end{theorem}

\noindent
\emph{Comment to the proof}. Any normalized formally elliptic operator of order four can be written in the form
(\ref{E:operatorL}), which turns out to be convenient for the computational purposes. Then one takes
the operator of order six $M:= 2 L^{\frac{3}{2}}_+$. The statement of the theorem follows from the analysis of the commutation relation $[L, M] = 0$, where one additionally has to rule out the rank one algebras  $\CC[L, M]$. \qed

\begin{remark} In the case $f'(0) = 0$, there are additional constraints between the coefficients of $f$ and Gr\"unbaum's parameters $K_{10}, K_{11}, K_{12}$ and $K_{14}$ (respectively, $K_2$ and $K_3$) to insure that the Laurent series
$c_2$ actually belongs to $\CC\llbracket z\rrbracket$. If that constraints are not satisfied, the resulting operators $L$ and $M$ still commute, but  the algebra $\CC[L, M]$ does not belong to $\gD$.
\end{remark}

\begin{remark} The different combinatorics of Gr\"unbaum's parameters  $c_0, c_1$ and $c_2$ in the formally self--adjoint and non--self--adjoint cases looks like artificial. However, this separation  turns out to be quite natural from the
point of view of the computation of the greatest common divisor $R_\chi$ for a character $\gB \stackrel{\chi}\lar \CC$. See also Remark \ref{R:Selfadjoint}. For the reader's convenience, and also following the work of Previato and Wilson \cite{PrW}, we decided to keep  Gr\"unbaum's notations \cite{Grun} in our article.

Although Gr\"unbaum's classification looks like quite massy on the first sight, it turns out to be perfectly suited
to describe
the spectral data $(X, p, \kF)$ of $\gB$ in   terms of Section \ref{S:SST}.
 Krichever and Novikov derived their formulae
\cite{KN} starting from the geometric side of Krichever's correspondence and then obtained
from it an explicit formula
for the operator $L$. A comparison between the answers of \cite{KN} and \cite{Grun} can be found in \cite[Section 6]{Grun}. At the present moment it is not clear to us,   how to generalize  the method of vector--valued Baker--Akhieser functions and deformations of Tyurin parameters of \cite{KN} on the case of  singular Riemann surfaces.
\end{remark}

\smallskip
\noindent
\textbf{Notation}. In the sequel, the following notation will be used.
\begin{itemize}
\item $\gB = \CC[L, M] \subset \gD$ is a genus one and rank two commutative subalgebra
with $L$ given by Gr\"unbaum's formulae from Theorem \ref{T:Gruenbaum}.
\item Next, $X$ is the compactified  spectral curve
of $\gB$, $p \in X$ is its point at infinity and $X_0 = X \setminus \{p\}$. If $X_0$ is singular then $s$ denotes its unique singular point.
\item Let $\kF$ be the spectral sheaf of $\gB$. See Corollary \ref{C:ListRankTwo} for a list of possibilities.
\item Finally, $\kT$ is  the Fourier--Mukai transform of $\kF$ and  $Z := \mathsf{Supp}(\kT) \subset X_0$.
\end{itemize}

\begin{proposition}\label{P:exponents} Let  $q = (\lambda, \mu)\in Z$  be such that
 $\kF$ is \emph{locally free} at $q$. Let
$\gB \stackrel{\chi}\lar \CC$ be the character corresponding to $q$ and
$$R_\chi := \partial^2 + c_1 \partial + c_2 = \gcd(L- \lambda, M- \mu) \in  \widetilde\gD.
$$
Let $\nu := - \mathrm{res}_0\bigl(c_1(z)\bigr) -1$ and
$\bigl(z^{\rho_1} w_1(z), z^{\rho_2} w_2(z)\bigr)$ be a basis of the solution space $\mathsf{Sol}(\gB, \chi) = \ker(R_\chi)$, where $0 \le \rho_1 < \rho_2 \in \NN_0$  and $w_i(0) \ne 0$ for $i = 1, 2$.
\begin{enumerate}
\item We have: $0 \le \nu  \le 3$ and $(\rho_1, \rho_2) \in \bigl\{(0,2), (0,3), (1, 2), (1, 3), (2, 3) \bigr\}$.
\item Next, $(\rho_1, \rho_2) = (2, 3)$ if and only if $q$ is a smooth point and
$\kF \cong \kO\bigl([q]\bigr) \oplus \kO\bigl([q]\bigr)$. This case occurs if and only if $\nu = 3$.
\item The case $\nu = 2$ is equivalent to   $(\rho_1, \rho_2) = (1, 3)$. If  $q$ is a smooth point then $\kF \cong \kA \otimes \kO\bigl([q]\bigr)$. If $q$ is singular then $\kF \cong \kB_{\bar{q}}$ for some smooth point $\bar{q} \in X$.
\end{enumerate}
\end{proposition}

\begin{proof} All essential ideas are taken from \cite{PrW}.

\smallskip
\noindent
(1) The indicial equation (\ref{E:indicialequation}) implies that $\rho_1 + \rho_2 =  \nu + 2$. By construction, $\ker(R_\chi) = \mathsf{Sol}(\gB, \chi) \subset \ker(L-\lambda)$. Recall that $\mathsf{ord}(L-\lambda) = 4$. If $z^\rho w(z) \in \ker(R_\chi)$ and $w(z) \ne 0$ then $\rho \le 3$ (by the uniqueness of solution of a differential equation with regular coefficients). All together, this implies the first statement.

\smallskip
\noindent
(2) Obviously, $\nu = 3$ if and only if  $(\rho_1, \rho_2) = (2, 3)$. However, $\{0, 1\} \cap \{\rho_1, \rho_2\} = \emptyset$ if and only if the map $\widetilde{\eta}_\chi$ from the commutative diagram (\ref{E:SolSpaces}) is zero.
Going through the list of vector bundles from  Corollary \ref{C:ListRankTwo} we conclude that the map
$\Gamma(X, \kF) \stackrel{\mathsf{ev}_q}\lar \kF|_q$ is zero
if and only if
$q$ is a smooth point and $\kF \cong \kO\bigl([q]\bigr) \oplus \kO\bigl([q]\bigr)$. See also \cite[Proposition 3.1]{PrW}.

\smallskip
\noindent
(3) If $q$ is a smooth point then the stated result is \cite[Theorem 1.2(ii)]{PrW}. If $q$ is singular, the result follows from Corollary \ref{C:ListRankTwo}.
\end{proof}

\subsection{Formally self--adjoint case} In this subsection, we describe the spectral sheaf of the algebra $\gB$ from Gr\"unbaum's Theorem \ref{T:Gruenbaum} in the formally self--adjoint case $c_1 = 0$.

\begin{lemma}\label{L:selfddeg} Let $L = \bigl(\partial^2 + \frac{1}{2} c_2\bigr)^2 + \gamma$ for some $c_2 \in \CC\llbracket z\rrbracket$ and $c_0 = \gamma \in \CC$ (degenerate self--adjoint case). Then $X$  is singular and  $\kF \cong \kS \oplus \kS$.
\end{lemma}

\begin{proof}
According to Gr\"unbaum \cite[Section 2]{Grun}, we have: $M = 2\bigl(\partial^2 + \frac{1}{2} c_2\bigr)^3 +
3 \gamma \bigl(\partial^2 + \frac{1}{2} c_2\bigr)$ and the equation of the spectral curve $X_0$  is
$
y^2 = 4 x^3 - 3 \gamma^2 x - \gamma^3.
$
Clearly, $X_0$ is singular at the point $s = (-\frac{\gamma}{2}, 0)$. Let $P = \Bigl(\partial^2 + \frac{1}{2} c_2\Bigr)$. It is easy to see that
$$
M = P \cdot \Bigl(L + \frac{\gamma}{2}\Bigr)
$$
implying that the order of the greatest common divisor $R_\chi$ (\ref{E:gcd}) for the character $\chi$ corresponding to the singular point $s$, is four. Therefore, we have:  $\kF\big|_s \cong \CC^4$. It remains to observe that $\kS \oplus \kS$ is the only
semi--stable sheaf or rank two and slope one on $X$, whose fiber over $s$ is four dimensional, see Corollary \ref{C:ListRankTwo}. Note that $\CC[L, M] \subset \CC[P]$, hence $\CC[L, M]$ is not maximal in this case.
\end{proof}

\begin{theorem}\label{T:selfadj} Let $L$ be given by  (\ref{E:operatorL}) with $c_1 = 0$ and $f' \ne 0$ (non--degenerate formally self--adjoint case). Then  $\kF$ is locally free.   Let  $\nu$ be  the order of vanishing of $f'(z)$ at $z = 0$. Then $Z$ is invariant under the involution $X_0 \stackrel{\imath}\lar X_0,\;
 \bigl((\lambda, \mu) \stackrel{\imath}\mapsto (\lambda, -\mu)\bigr)$ and the following results are true
 (we assume that  $X_0 = V\Bigl(y^2 - 4 x^3  - 2 K_3  x + \dfrac{K_2}{2}\Bigr)$ is singular):
\begin{enumerate}
\item If $\nu = 0$ then $\kF$ is isomorphic to
\begin{enumerate}
\item $\kO\bigl([q]\bigr) \oplus \kO\bigl([\imath(q)]\bigr)$ if  $Z = \left\{q, \imath(q) \right\}$ with $q \ne \imath(q)$.
\item $\kA \otimes \kO\bigl([q]\bigr)$ if $Z = \{q\} = \{\imath(q)\}$ and $q \ne s$.
\item $\kB_p$ if $Z = \{s\}$.
\end{enumerate}
\item If $\nu = 1$  then $\kF \cong \kO\bigl([q]\bigr) \oplus \kO\bigl([\imath(q)]\bigr)$ with $q \ne \imath(q)$ and $Z = \left\{q, \imath(q) \right\}$.
\item If $\nu = 2$ then necessarily  $Z = \{q\}$ with $q = \imath(q)$.
\begin{enumerate}
\item If $q \ne s$ then $\kF \cong \kA \otimes \kO\bigl([q]\bigr)$.
\item If $q = s$ then $\kF \cong \kB_p$.
\end{enumerate}
\item If $\nu = 3$ then $\kF \cong \kO\bigl([q]\bigr)\oplus \kO\bigl([q]\bigr)$, where $q = \imath(q)$ is
a smooth point of $X_0$. In this case, $Z = \{q\}$, what  can occur  only if $X_0$ is nodal.
\end{enumerate}
\end{theorem}

\begin{proof} Let $q = (\lambda, \mu) \in X_0$ and $\gB \stackrel{\chi}\lar \CC$ be the corresponding character.
The key point is the following result \cite[Section 5]{PrW}: there exist $R, Q \in \gD$ both of order two such that
$$
M - \mu = Q \cdot (L - \lambda) + R,
$$
where $R = a_0 \partial^2 + a_1 \partial + a_2$ with $a_0 = (2\lambda + f)$ and $a_1 = - f'$. Since $f$ is not a constant, the order of $R_{\chi_q}$ is two for all $q \in X_0$ implying that the spectral sheaf $\kF$ is locally free. Note that $\nu$ coincides  with the parameter introduced in Proposition \ref{P:exponents}.

By Theorem \ref{T:SupportSpecSheaf} we have: $Z = \bigl\{(\lambda_0, \pm \mu_0)\bigr\}$, where $\lambda_0 = -\frac{1}{2} f(0)$ and $\pm \mu_0$ are the roots of the equation $\mu^2 = h(\lambda_0)$ with $h(\lambda) = 4 \lambda^3 + 2K_3 \lambda - \frac{1}{2} K_2$. Unless $Z = \{s\}$, the description of $\kF$  can be obtained along the same lines as in \cite[Theorem 1.2]{PrW}, see also Proposition \ref{P:exponents}. From now on we assume that $Z = \{s\}$. According to Corollary \ref{C:ListRankTwo}, $\kF \cong \kB_{\bar{q}}$ for some smooth point $\bar{q} \in X$ and we only have to show that $\bar{q} = p$. Note that
$$
\mathrm{res}_0\Bigl(\frac{f'(z)}{f(z)-f(0)}\Bigr) = \nu + 1.
$$
Proposition \ref{P:exponents} implies that $0 \le \nu \le 3$.

\smallskip
\noindent
\underline{Case 1}. Assume that Gr\"unbaum's parameters $K_2, K_3$ and $f$ are such that $f'(0) \ne 0$ (i.e.~$\nu = 0$). Let $\gB$ be the corresponding commutative subalgebra of $\gD$. Consider now the $\CC[t]$--flat family $\gB_B \subset
\bigl(\CC[t]\bigr)\llbracket z\rrbracket[\partial]$ defined by the Gr\"unbaum's parameters $K_2, K_3(t) := K_3  + t$ and $f$.
Let $X_B \stackrel{\pi}\lar B$ be the corresponding spectral fibration (here, $B = \Spec\bigl(\CC[t]\bigr)$) and $\kF_B$ be the corresponding spectral sheaf, see Theorem \ref{T:relativeSpectData}. For any $b \in B$ we denote by $X_b = \pi^{-1}(b)$ the fiber over $b$ and $\kF_b := \kF_B\Big|_{X_b}$. Clearly, $\kF_0 \cong \kF$ and $\kF_b \cong \kO_{X_b}\bigl([q_1(b)] + [q_2(b)]\bigr)$ for $b \ne 0$ from some open neighbourhood $U \subset B$ of $0$, where $\imath(q_1(b)) = q_2(b)$ in $X_b$.  Therefore,
$\det\bigl(\kF_b\bigr) \cong \kO_{X_b}\bigl(2[p]\bigr)$ for all $b \in U\setminus \{0\}$. But then we also have:
$\det(\kF_0) \cong \kO_{X_0}\bigl(2[p]\bigr)$ and therefore $\kF \cong \kB_p$.

\smallskip
\noindent
\underline{Case 2}. Assume that Gr\"unbaum's parameters $K_2, K_3$ and $f$ are such that $f'(0) =  0$. Then $f$ has an expansion of the form
$
f(z) = \alpha + \beta z^2 + \gamma z^3 + \delta z^4 + \dots
$
Now we have to use the fact that  Gr\"unbaum's parameter $c_2(z)$ given by (\ref{E:coefc2}) is regular. This in particular implies that $\alpha^3 + 2K_3 \alpha + K_2 + 2\beta^2 = 0$, i.e. $\Bigl(-\dfrac{\alpha}{2}, \pm \beta\Bigr) \in X_0$. Since we assumed that $\kT$ is supported at the singular point of $X_0$, we get: $\beta = 0$. Hence, $\nu \ge 2$ and in virtue of Proposition \ref{P:exponents} we have: $\nu = 3$, i.e.~$\gamma \ne 0$. Requiring  the regularity of $c_2(z)$, we get the following constraint:
$
2K_3 + 3 \alpha^2 - 24 \delta = 0.
$
 Observe that the point  $\bigl(-\dfrac{\alpha}{2}, 0\bigr) \in X_0$ is singular if and only if $\delta = 0$. Summing up, we have in this case:
\begin{equation}\label{E:explicit}
\left\{
\begin{array}{cll}
f & = &  \alpha + \gamma z^3 + \tau z^5 + \dots, \;\mbox{\rm with}\;  \gamma \ne 0,\\
K_2 & = & 2 \alpha^3,\\
K_3 & = &  - \frac{3}{2}\alpha^2.
\end{array}
\right.
\end{equation}
Let $\gB = \CC[L, M]$ be the corresponding commutative subalgebra of $\gD$. It admits the following \emph{flat} deformation $\gB_B$ over the base $B = \Spec\bigl(\CC[\delta]\bigr)$:
\begin{equation}\label{E:explicit2}
\left\{
\begin{array}{cll}
f(\delta) & = &    f + \delta z^4, \\
K_2(\delta) & = & 2 \alpha^3 - 24 \alpha \delta,\\
K_3(\delta) & = &  12 \delta - \frac{3}{2}\alpha^2.
\end{array}
\right.
\end{equation}
The total space of the corresponding genus one fibration $X_B \stackrel{\pi}\lar B$ is given by the equation
$$
 X_B:= \overline{V\bigl(y^2 - 4 x^3 - (24 \delta - 3 \alpha^2) x   - 2 \alpha (\alpha^2 - 12 \delta)\bigr)} \subset
\mathbbm{P}^2_{(x, y)} \times \mathbb{A}^1_\delta.
$$
It is interesting to note that $X_B$
is singular and $B \stackrel{\sigma}\lar X_B,  \delta \mapsto \Bigl(-\dfrac{\alpha}{2}, 0, \delta\Bigr)$ is a section of $\pi$.
Let $\kF_B$ be the spectral sheaf of $\gB_B$, see Theorem \ref{T:relativeSpectData}.
There exists an open subset $U \subset B$ with $0 \in U$ and such that for all $b \in U \setminus \{0\}$ we have:
$\kF_b:= \kF_B\big|_{Y_b} \cong \kA \otimes \kO\bigl([q]\bigr)$ with $q = \Bigl(-\dfrac{\alpha}{2}, 0\Bigr)$ for $b \ne 0$. Therefore, $\det(\kF_b) \cong \kO\bigl(2[p]\bigr)$ for $b \in U \setminus \{0\}$. This implies that
$\det(\kF_0) \cong \kO\bigl(2[p]\bigr)$ as well, see Proposition \ref{P:continuitybundles}. Thus, $\kF \cong \kB_p$ as claimed.
\end{proof}

\begin{example}
Let $\gB = \CC[P, Q]$ be as in the  example of Dixmier (\ref{E:Dixmier}) for $\kappa = 0$. Then the spectral sheaf of $\gB$ is $\kB_p$.
\end{example}

\subsection{Non--self--adjoint case} Let $L$ be the fourth order differential operator given by Gr\"unbaum's parameters $K_{10}, K_{11}, K_{12}, K_{14}$ and $f$ as in  (\ref{E:GruenbaumCoeff}). The equation of the affine spectral curve $X_0$  of the algebra $\CC[L, M]$ is $y^2 = 4 x^3 - g_2 x - g_3$ with
\begin{equation}\label{E:WeierConstants}
\left\{
\begin{array}{ccl}
g_2  & = &  3 K_{12}^2  +  K_{10} K_{11} - K_{14},\\
g_3  & = & \frac{1}{4}\bigl(2 K_{10} K_{11} K_{12} + 4 K_{12}^3 + K_{14}(K_{11}^2 + 4 K_{12}) - K_{10}^2\bigr).
\end{array}
\right.
\end{equation}
For any $\lambda \in \CC$ pose
\begin{equation}\label{E:abc}
\left\{
\begin{array}{ccl}
a(\lambda) & = & \bigl(\lambda + \frac{1}{2}K_{12}\bigr)^2 + \frac{1}{4} K_{14} \\
b(\lambda) & = & \bigl(\lambda + \frac{1}{2}K_{12}\bigr) K_{11} - \frac{1}{2} K_{10} \\
c(\lambda) & = & -\lambda + K_{12} + \frac{1}{4} K_{11}^2.
\end{array}
\right.
\end{equation}

\smallskip
\noindent
Our analysis  of the spectral sheaf $\kF$ is based in the following result from the article of Previato and Wilson
\cite[Section 5]{PrW} attributed there to the PhD thesis of Latham \cite{LathamPhD} (see also \cite{Latham}).

\begin{theorem}\label{T:SpecSheafNotLF}
Let $(\lambda, \mu)$ be  any point of $X_0$ (smooth or singular) and  $\gB \stackrel{\chi}\lar \CC$ be the corresponding character. Let $\widetilde{R}_\chi, \widehat{R}_\chi \in \gD$ be the differential operators defined by the following
conditions:
\begin{equation}\label{E:computinggcd}
\left\{
\begin{array}{ccll}
M-\mu & = & \widetilde{Q}_\chi \cdot (L-\lambda) + \widetilde{R}_\chi, & \mathsf{ord}(\widetilde{R}_\chi) \le 3\\
L-\lambda & = & \widehat{Q}_\chi \cdot \widetilde{R}_\chi  + \widehat{R}_\chi, & \mathsf{ord}(\widehat{R}_\chi) \le 2.
\end{array}
\right.
\end{equation}
Then we have: $\mathsf{ord}(\widetilde{R}_\chi) =  3$ and
$
\widehat{R}_\chi = e_0(z; \lambda, \mu) \partial^2  - e_1(z; \lambda, \mu) \partial + e_2(z; \lambda, \mu)
$
with
\begin{equation}\label{E:CoeffOfgcd}
e_0 = a(\lambda) + b(\lambda) f + c(\lambda) f^2 \quad \mbox{\rm and} \quad
e_1 = \frac{1}{2}\bigl(b(\lambda) -\mu\bigr)f + c(\lambda) ff'.
\end{equation}
\end{theorem}

\noindent
Similarly to  \cite[Section 5]{PrW}, we have the following result.

\begin{proposition}\label{P:SuportSpecSheaf}
A point $q = (\lambda, \mu) \in X_0$ belongs to $Z$  if and only if  $a(\lambda) = 0$ and $\mu = - b(\lambda)$.
\end{proposition}

\begin{proof} A lengthy but elementary computation allows to show the following

\smallskip
\noindent
\underline{Fact}. If a point $(\lambda, \mu) \in \CC^2$ belongs to $X_0$ and
$a(\lambda) = 0$ then necessarily $\mu = \pm b(\lambda)$.

\smallskip
\noindent
Let $R_\chi := \mathsf{gcd}(L-\lambda, M - \mu)$ in the sense of Theorem \ref{T:gcd}.

\smallskip
\noindent
\underline{Case 1}. Assume that $a(\lambda) = b(\lambda) = c(\lambda) = 0$. Then $e_0(z; \lambda, \mu) = 0$. In virtue of the formulae (\ref{E:computinggcd}) we see that $\mathsf{ord}(\widehat{R}_\chi) \le 1$ in this case. However,
$\rk\bigl(\CC[L, M]\bigr) = 2$ and the only possibility for this to be true is that $\widehat{R}_\chi = 0$.
Hence, $\mathsf{ord}(R_\chi) = 3$.
This case occurs if and only if
$X_0$ is singular with the singular point $s= (\lambda, 0)$ and $\kF$ is not locally free at $s$. See Theorem  \ref{T:notLF} below. In this case, the singular point $s$ belongs to the support of $\kT$ due to
 Theorem \ref{T:SupportSpecSheaf}.

\smallskip
\noindent
\underline{Case 2}. Assume now that $\bigl(a(\lambda), b(\lambda), c(\lambda)\bigr) \ne (0, 0, 0)$. In this case,
$e_0(z; \lambda, \mu) \ne 0$ and
$$
R_\chi = \frac{1}{e_0(z; \lambda, \mu)} \widehat{R}_\chi = \partial^2 - \frac{e_1(z; \lambda, \mu)}{e_0(z; \lambda, \mu)} \partial + \frac{e_2(z; \lambda, \mu)}{e_0(z; \lambda, \mu)}.
$$
According to Theorem \ref{T:SupportSpecSheaf}, $(\lambda, \mu)$ belongs to the support of $\kT$ if and only if
the Laurent power series $\dfrac{e_1(z; \lambda, \mu)}{e_0(z; \lambda, \mu)}$ has a pole at $z = 0$. Taking into account
explicit expressions (\ref{E:CoeffOfgcd}) for $e_i(z; \lambda, \mu)$ for $i = 0, 1$ as well as the assumption
$f(0) = 0$, we see that $a(\lambda) = 0$. Therefore, $\mu = \pm b(\lambda)$. Note, that by assumption
$\bigl(b(\lambda), c(\lambda)\bigr) \ne (0, 0)$.

\smallskip
\noindent
If $\mu = -b(\lambda)$ then
$
\dfrac{e_1(z; \lambda, \mu)}{e_0(z; \lambda, \mu)} = \dfrac{f'(z)}{f(z)}.
$
This function has a pole at $z = 0$ as $f(0) = 0$. Therefore, the point $\bigl(\lambda, -b(\lambda)\bigr)$ belongs to the support of $\kT$ due to Theorem \ref{T:SupportSpecSheaf}. Moreover, the order of vanishing  of  $f$ at $0$ is at most four, see Proposition \ref{P:exponents}.

\smallskip
\noindent
Now suppose that $\mu = b(\lambda)$. Then $
\dfrac{e_1(z; \lambda, \mu)}{e_0(z; \lambda, \mu)} = \dfrac{c(\lambda) f'(z)}{b(\lambda) + c(\lambda)f(z)}$ has a pole at $z = 0$ if and only if $b(\lambda) = 0$ (and we are in the previous case).
\end{proof}

\begin{theorem}\label{T:notLF}
Let $\kB = \CC[L, M]$ be a genus one and rank two commutative subalgebra, which is not formally self--adjoint and given by Gr\"unbaum's parameters $K_{10}, K_{11}, K_{12}, K_{14}$ and $f$. Then
we have:
\begin{enumerate}
\item   the spectral sheaf $\kF$ of $\gB$ is not locally free if and only if
\begin{equation}\label{E:notlocafree}
\left\{
\begin{array}{ccl}
K_{10} & = & \; \; \, (3 K_{12} + \frac{1}{2} K_{11}^2) K_{11}\\
K_{14} & =  & - (3 K_{12} + \frac{1}{2} K_{11}^2)^2.
\end{array}
\right.
\end{equation}
\item Moreover, in this case $\kF$ is indecomposable (i.e.~isomorphic to $\kU_\pm$ in the nodal case, respectively to $\kU$ in the cuspidal case) if and only if
\begin{equation}\label{E:DiscrimTF}
\Delta:= 6 K_{12} + K_{11}^2 = 0.
\end{equation}
\item If $\Delta \ne 0$ then  $\kF \cong \kS \oplus  \kO\bigl([q]\bigr)$, where
$q = \bigl(-2 K_{12} - \frac{1}{4} K_{11}^2, -\frac{1}{2}K_{11}(K_{11}^2 + 6K_{12})\bigr)$.
\end{enumerate}
\end{theorem}

\begin{proof} (1) Assume that  $\kF$ is not locally free.
According to Theorem \ref{T:SupportSpecSheaf}, this is equivalent to  $\mathsf{ord}(R_\chi) = 3$, where $\chi$ is the character, corresponding to some point $(\lambda_0, 0) \in X_0$. This can happen if and only if $\widehat{R}_\chi = 0$. In particular, $e_0(z; \lambda_0, 0) = 0$ implying that
$a(\lambda_0) = b(\lambda_0) = c(\lambda_0)=0$. From the equality $a(\lambda_0) = 0$ we get $\lambda_0 = K_{12} + \frac{1}{4} K_{11}^2$, whereas the vanishings $b(\lambda_0) = c(\lambda_0) = 0$ imply the constraints (\ref{E:notlocafree}).

Other way around, assume that  (\ref{E:notlocafree}) are satisfied. A direct  computation shows that the Weierstra\ss{} parameters $g_2, g_3$ given by  the formulae (\ref{E:WeierConstants}), take the following  form:
\begin{equation*}
\left\{
\begin{array}{ccl}
g_2  & = & \;   3\bigl(2 K_{12} + \frac{1}{2} K_{11}^2\bigr)^2\\
g_3  & = &  -\bigl(2 K_{12} + \frac{1}{2} K_{11}^2\bigr)^3.
\end{array}
\right.
\end{equation*}
By Theorem \ref{T:basicsGenusOne}, the spectral curve $X_0$ is singular with the singular point $s = (\lambda_0, 0)$, where $\lambda_0 = K_{12} + \frac{1}{4} K_{11}^2$. Moreover, constraints (\ref{E:notlocafree}) imply that
$a(\lambda_0) = b(\lambda_0) = c(\lambda_0) = 0$, hence $\kF$ is indeed not locally free at $s$.

\smallskip
\noindent
(2) The possibilities for the  spectral sheaf $\kF$  are listed in Corollary \ref{C:ListRankTwo}. The case
$\kF \cong \kS \oplus \kS$ is excluded since $\mathsf{ord}(\widetilde{R}_\chi) = 3$ by Theorem \ref{T:SpecSheafNotLF}, implying that $\dim_{\CC}\bigl(\kF\big|_{s}\bigr) \le 3$. Hence, $\kF$ is indecomposable if and only if $\kT$ is supported at the singular point of $X_0$.
 According to Proposition \ref{P:SuportSpecSheaf}, this occurs if and only if $K_{14} = 0$: otherwise, the equation
 $a(\lambda) = 0$ has two different solutions, both contributing  to the support of $\kT$ due to Proposition \ref{P:SuportSpecSheaf}. Since we already showed that the formulae (\ref{E:notlocafree}) are true, the indecomposability of $\kF$ is equivalent to the vanishing $\Delta = 0$.

\smallskip
\noindent
(3) Assume that the equations  (\ref{E:notlocafree}) are satisfied and $\Delta \ne 0$. Then the equation
$a(\lambda) = 0$ has two different solutions: $\lambda_0 = K_{12} + \frac{1}{4} K_{11}^2$ and
$\tilde{\lambda}_0 = -2K_{12} - \frac{1}{4} K_{11}^2$. The torsion sheaf  $\kT$ is supported at $s = (\lambda_0, 0)$ and $q:= \bigl(\tilde{\lambda}_0, -b(\tilde\lambda_0)\bigr) = \bigl(-2 K_{12} - \frac{1}{4} K_{11}^2, -\frac{1}{2}K_{11}(K_{11}^2 + 6K_{12})\bigr)$. Taking into account Corollary \ref{C:ListRankTwo}, we get the statement.
\end{proof}

\begin{lemma}\label{L:ComputSeries} Let $g \in \CC\llbracket z\rrbracket$. Then the  Laurent series
$
h = \dfrac{2 g g'' - g'^2}{g^2}
$
is regular at $z = 0$ if and only if $g(0) \ne 0$ or $g(z) = z^2 \tilde{g}(z)$ with $\tilde{g}(0) \ne 0$ and
$\tilde{g}'(0) = 0$.
\end{lemma}

\begin{proof}
Obviously, $h(z)$ is regular provided $g(0) \ne 0$. Assume that $g(z) = z^\rho \tilde{g}(z)$ with $\rho \in \NN_0$
and
$\tilde{g}(0) \ne 0$. Note that
\begin{equation}\label{E:expansions}
h = \frac{g''}{g} + \Bigl(\frac{g'}{g}\Bigr)' = \Bigl(\frac{\rho(\rho-1)}{z^2} + \frac{2 \rho}{z} \frac{\tilde{g}'}{\tilde{g}} + \varphi\Bigr) +
 \Bigl(-\frac{\rho}{z^2} + \psi\Bigr)
\end{equation}
for appropriate   $\varphi, \psi \in \CC\llbracket z\rrbracket$. If $\rho \ge 1$ then $h$ is regular if and only if $\rho = 2$ and $\tilde{g}'(0) = 0$. Therefore, the series $g(z)$ has the form
\begin{equation}\label{E:seriesformelle}
g(z) = \zeta_2 z^2 + \sum\limits_{i = 4}^\infty \zeta_i z^i \quad \mbox{\rm with} \; \zeta_2 \ne 0.
\end{equation}\end{proof}
\begin{corollary}\label{C:OperatorsNotLF}
Let $\gB = \CC[L, M] $ be  a genus one and rank two commutative subalgebra in $\gD$. Then the  spectral sheaf of $\gB$ is indecomposable and not locally free (i.e~isomorphic to $\kU_\pm$ in the nodal case and to $\kU$ in the cuspidal case) if and only if $L$  is formally non--self--adjoint and  given by the formulae (\ref{E:operatorL}) with the  parameters $c_0, c_1$ and $c_2$:
\begin{equation}\label{E:GruenbaumCoeffNotLF}
    \left\{
    \begin{array}{ccl}
    c_0 & = & -f^2 + \varrho f - \dfrac{\varrho^2}{6} \\
    c_1 & = & f' \\
    c_2 & = & \dfrac{2\varrho f^3 -\varrho^2 f^2  - f^4 + f''^2 - 2 f' f'''}{2 f'^2}
    \end{array}
    \right.
\end{equation}
for an arbitrary  $\varrho \in \CC$ and any
$f \in \CC\llbracket z\rrbracket$ satisfying $f(0) = 0$ and either of two conditions:
\begin{itemize}
\item $f'(0) \ne 0$ or
\item $f'(0) = f''(0) = f^{(4)}(0) = 0$, $f'''(0) \ne 0$.
\end{itemize}
The equation of the spectral curve in this case is
\begin{equation}
y^2 = 4 x^3 - \frac{1}{12}\varrho^4 x + \frac{1}{216} \varrho^6.
\end{equation}
\end{corollary}

\begin{remark} In the notation of Theorem \ref{T:Gruenbaum} we have $\varrho = K_{11}$. Note that the
family (\ref{E:GruenbaumCoeffNotLF}) admits an obvious involution $\varrho \mapsto - \varrho$. It turns out that
this involution  corresponds to the flip $\kU_\pm \mapsto \kU_\mp$ on the level of spectral sheaves. The precise description  of $\kF$ in the nodal case (i.e.~$\kU_+$ versus $\kU_-$) is  rather subtle,  see the proof of
Theorem \ref{T:interestingfamily}.
\end{remark}

\begin{example}\label{E:funny} Let us set  $\varrho = 0$ and $f = z$ in the equations (\ref{E:GruenbaumCoeffNotLF}). Then we get
$$
L = \Bigl(
\partial^2 - \frac{z^4}{4}
\Bigr)^2 + 2 \partial - z^2.
$$
A straightforward  computation shows that in this case $M:=2 L^{\frac{3}{2}}_+$ is given by the formula
$
M = 2\partial^6 - \frac{3}{2} z^4 \partial^4 + 6(1-2z^3)\partial^3 + z^2\bigl(\frac{3}{8} z^6 - 45\bigr)\partial^2  +
z\bigl(3z^6 - \frac{3}{2}z^3 - 54\bigr)\partial + \bigl(-\frac{1}{32} z^{12} + \frac{37}{4} z^6 - 3 z^3-14\bigr).
$
Moreover, another straightforward computation yields:
$$
R := \mathrm{gcd}(L, M) = \partial^3 -\frac{1}{2} z^2 \partial^2 + z\Bigl(-\frac{1}{4} z^3+1\Bigr) \partial
+
\Bigl( \frac{1}{8} z^6 - \frac{3}{2} z^3 + 1\Bigr).
$$
Since $\mathsf{ord}(R) = 3$, the spectral sheaf of $\CC[L, M]$ is the torsion free sheaf $\kU$, as predicted. Notably,  the coefficients of  $R$ are regular. We hope that a more detailed treatment  of genus one commutative subalgebras in the Weyl algebra $\mathfrak{W} = \CC[z][\partial]$
with a cuspidal  spectral curve and the spectral sheaf which is not locally free will be helpful  for  various studies related to Dixmier's conjecture about $\Aut(\mathfrak{W})$, see \cite{MironovZheglov}.
\end{example}

\smallskip
\noindent
The following result characterizes those genus one and rank two commutative subalgebras of $\gD$, whose spectral curve $X$ is singular and the associated torsion sheaf $\kT$ is indecomposable and supported at the singular point of $X$.

\begin{theorem}\label{T:interestingfamily}
Let $\kB = \CC[L, M]$ be given by Gr\"unbaum's parameters $K_{10}, K_{11}, K_{12}, K_{14}$ and $f$. Then
 the following results are true.
\begin{enumerate}
\item The (affine) spectral curve $X_0$ of $\gB$ is singular and the torsion sheaf $\kT$ is supported at the singular point of $X_0$ if and only if $K_{10} = 0 = K_{14}$. In this case, $X_0 = \Spec(R)$ with
\begin{equation}
R = \CC[x, y]\big/\Bigl(y^2  - 4\Bigl(x + \frac{K_{12}}{2}\Bigr)^2 (x - K_{12})\Bigr).
\end{equation}
\item The spectral sheaf $\kF$ of $\gB$ is locally free if and only if $\Delta:= 6 K_{12} + K_{11}^2 \ne  0$. In this case,
$\kF \cong \kB_{\bar{q}}$ with
\begin{equation}\label{E:CoeffPoint}
\bar{q} = \Bigl(\frac{1}{4} K_{11}^2 + K_{12}, \frac{K_{11}}{4}\bigl(6K_{12} + K_{11}^2\bigr)\Bigr).
\end{equation}
\item Moreover, for the Fourier--Mukai transform $\kT$ of $\kF$ we have:
\begin{equation}\label{E:DescrTSheaf}
\kT \cong \widehat{R}\big/\Bigl(\bigl(x + \frac{K_{12}}{2}\bigr)^2, y - K_{11}\bigl(x + \frac{K_{12}}{2}\bigr)\Bigr).
\end{equation}
\end{enumerate}
\end{theorem}

\begin{proof} (1) According to Proposition \ref{P:SuportSpecSheaf}, the support of $\kT$ consists of a single point $q = (\lambda_0, \mu_0)$ if and only if $K_{14} = 0$. In this case, $\lambda_0 = -\frac{1}{2} K_{12}$ and $\mu_0 = - b(\lambda_0)$. If $q$ is the singular point of $X_0$ then $b(\lambda_0) = 0$ implying that $K_{10} = 0$. Other way around, if $K_{10} = 0 = K_{14}$ then $X_0$ is given by the equation
$
y^2 = 4 x^3 - 3 K_{12}^2 x - K_{12}^3.
$
According to Theorem \ref{T:basicsGenusOne}, the curve $X_0$ is singular with the singular point $s = (\lambda_0, 0) = \bigl(-\frac{1}{2} K_{12}, 0\bigr)$. Moreover, $a(\lambda_0) = 0 = b(\lambda_0)$, hence $\kT$ is indeed supported at $s$.

\smallskip
\noindent
(2) We already showed in Theorem \ref{T:notLF} that the spectral sheaf $\kF$ is locally free if and only if
$\Delta \ne 0$. Therefore, in this case $\kF \cong \kB_{\bar{q}}$ for some smooth point
$\bar{q} \in X$ determined by the condition $\det(\kF) \cong \kO\bigl([p]+[\bar{q}]\bigr)$. To compute the determinant of $\kF$, we use again a deformation argument.

\smallskip
\noindent
\underline{Case 1}. Assume that $f'(0) \ne 0$. Then the power series $c_2$ given by  (\ref{E:GruenbaumCoeff}) is automatically regular and the non--zero Gr\"unbaum's parameters $K_{11}, K_{12}$ are independent of the  coefficients of the power series $f$. Keeping $K_{11}, K_{12}$ and $f$ unchanged and introducing new parameters
$\alpha = K_{10}$ and $\beta = K_{14}$, we get a family $\gB_{B}$ of commutative subalgebras
in $\gD$ given by (\ref{E:GruenbaumCoeff}), flat over the base $B = \Spec(\CC[\alpha, \beta])$ and such that $\gB_{(0, 0)} \cong \gB$. Let $\kF_B$ be the corresponding spectral sheaf, see Theorem \ref{T:relativeSpectData}. Assume that  $t = (\alpha, \beta) \in B$ is such that the support of the Fourier--Mukai transform $\kT_{t}$ of the  corresponding spectral sheaf $\kF_{t}$ is locally free and supported at two different points of the spectral curve $X_{t}$. According to Proposition \ref{P:SuportSpecSheaf},
the support of $\kT_{t}$ is $\bigl\{q_1, q_2\bigr\} = \bigl\{(\lambda_1, -b(\lambda_1), (\lambda_2, -b(\lambda_2)\bigr\}$, where $\lambda_1$ and $\lambda_2$ are the roots of the equation
$
\lambda^2 + K_{12} \lambda + \frac{1}{4}(K_{12}^2 + \beta) = 0.
$
Moreover, $\kF_{t} \cong \kO\bigl([q_1]\bigr) \oplus \kO\bigl([q_2]\bigr)$, hence
$
\det\bigl(\kF_{t}\bigr) \cong \kO\bigl([q_1]+[q_2]\bigr) \cong \kO\bigl([p]+[\bar{q}]\bigr),
$
where $\bar{q} = q_1+ q_2$ with $``+"$ taken in  the sense of the group law  on the set of smooth points
of $X_{t}$. Computing explicitly $q_1 + q_2 \in X_{t}$ and  then setting $\alpha = \beta = 0$, we get:
$\det(\kF) \cong \kO\bigl([p]+[\bar{q}]\bigr)$ with $\bar{q}$ given by (\ref{E:CoeffPoint}).

\smallskip
\noindent
\underline{Case 2}. Suppose now that $f'(0) =  0$. The proof in this  case is analogous to the previous one, but is technically more involved. First note that the Laurent series $\dfrac{f}{f'}$ is regular at $z = 0$. Since $K_{10} = K_{14} = 0$, the regularity of $c_2$ given by (\ref{E:GruenbaumCoeff}) is equivalent to the regularity of
$\dfrac{f''^2 - 2 f' f'''}{f'^2}$. Lemma \ref{L:ComputSeries} implies that the order of vanishing of  $f$ at $z = 0$ is precisely three. Moreover, $f$ has the following form:
$
f = \xi_3 z^3 + \sum\limits_{i= 5}^\infty \xi_i z^i
$
with $\xi_3 \ne 0$, see (\ref{E:seriesformelle}). Setting
$
\left\{
\begin{array}{ccl}
f_\xi & = & f + \xi z^4 \\
K_{10} & = & - 24 \xi
\end{array}
\right.
$
and keeping  the parameters $K_{11}, K_{12}$ untouched, then we get a flat family of commutative subalgebras $\gB_{B}$ over the base
$B = \Spec(\CC[\xi])$ with $\kB_0 \cong \gB$. As $K_{14} = 0$, the spectral sheaf $\kF_\xi$ is isomorphic to $\kA \otimes \kO\bigl([q_\xi]\bigr)$ for $\xi \ne 0$,  where $q_\xi$ is a smooth point of the spectral curve (such behaviour is completely parallel to the self--adjoint case, see the proof of Theorem \ref{T:selfadj}). Therefore,
$\det(\kF_\xi) \cong \kO\bigl(2[q_\xi]\bigr)$.
In a similar manner we get again: $\kF = \kB_{\bar{q}}$ with $\bar{q}$
given by (\ref{E:CoeffPoint}).

\smallskip
\noindent
(3) In the case $\Delta \ne 0$, the isomorphism (\ref{E:DescrTSheaf}) for the torsion sheaf $\kT$ follows from Theorem \ref{T:classificationRankTwo}. It remains to describe  $\kT$ in the
case when  $\Delta = 0$ and the spectral curve is nodal. Assume that $K_{10} = K_{14} = 0$,  $K_{12} = \tau$ is fixed and $K_{11} = \theta$ can be varied. Furthermore, let $f \in z\CC\llbracket z\rrbracket$ be such that $c_2$ is regular at $z = 0$. Then we get a family of commutative subalgebras $\gB_T$ flat over $T = \Spec(\CC[\theta])$, whose affine spectral surface   $X_T \subset \mathbb{A}^2_{x,y} \times T$ is given by the  equation
$
y^2 = 4 \Bigl(x + \frac{\tau}{2}\Bigr)^2 (x-\tau).
$
Let $\kF_T$ be the spectral sheaf of this family and $\kT_T$ its relative Fourier--Mukai transform. Let $b \in \CC = T$ be such that $b^2 + 6 \tau = 0$. Clearly, the torsion sheaf
$\bigl(\kT_T\bigr)\big|_{X_b}$ is a quotient of $\kO_{X_b}$. Therefore, $\kT_{T_0}:= \kT_T\big|_{T_0}$ is a quotient
of $\kO_{X_{T_0}}$ for some open neighbourhood $T_0 \subset T$ of  $b$. Using Remark \ref{R:HilbertSchemes} as well as the universal property of the Hilbert scheme of points applied to $(T_0, \kT_{T_0})$, we get  a uniquely determined  morphism  $T_0 \stackrel{\gamma}\lar \PP^1, \theta \mapsto \bigl(\gamma_0(\theta): \gamma_1(\theta)\bigr)$ such that
$$
\kT_\theta \cong \widehat{R}/\bigl(\bigl(x + \frac{\tau}{2}\bigr)^2, \gamma_0(\theta)\bigl(x + \frac{\tau}{2}\bigr)- \gamma_1(\theta)y \bigr).
$$
 From part (2) and Theorem \ref{T:classificationRankTwo}(2) we already know that for $\theta \in T_0\setminus \{b\}$ we have:  $\gamma(\theta) = (1: \theta)$. By continuity of $\gamma$ we finally obtain: $\gamma(b) = (1: b)$. Theorem is proven.
\end{proof}

\begin{remark}
The description of the spectral sheaf $\kF$ of the algebra $\gB$ in the case $s \notin Z$  is the same as in the work of Previato and Wilson \cite{PrW}. In particular, $q = (\lambda, \mu)$ belongs to $Z$ if and only if $a(\lambda) = 0$ and $\mu = - b(\lambda)$. There are namely the following  possibilities:
\begin{itemize}
\item $\kF \cong \kO\bigl([q]\bigr) \oplus \kO\bigl([q]\bigr)$ if $Z = \bigl\{q, q'\bigr\}$.
\item $\kF \cong \kO\bigl([q]\bigr) \otimes \kA$ or $\kF \cong \kO\bigl([q]\bigr) \oplus \kO\bigl([q]\bigr)$ if $Z = \bigl\{q\bigr\}$. The last case occurs if and only if $f$ has a zero of order four at $z = 0$.
\end{itemize}
\end{remark}

\subsection{Spectral sheaf of the Fourier transform of Dixmier's example} The methods developed in our article
can be applied to determine the spectral sheaves of genus one and rank three commutative subalgebras of $\gD$.

\begin{example}
The Weyl algebra $\mathfrak{W} = \CC[z][\partial]$ admits an algebra automorphism $z \stackrel{\phi}\mapsto \partial, \partial \stackrel{\phi}\mapsto -z$ called Fourier transform. Consider now the Fourier transform of Dixmier's example (\ref{E:Dixmier}). Namely, for any
 $\kappa \in \CC$,
put $\widehat{D} := \phi(D) = \partial^3 + z^2 + \kappa$ and pose
\begin{equation}\label{E:FTDixmier}
\widehat{P} := \phi(P) = \widehat{D}^2 + 2\partial  \quad \mbox{and} \quad
\widehat{Q} := \phi(Q) = \widehat{D}^3 + \frac{3}{2}\bigl(\partial \widehat{D} + \widehat{D} \partial\bigr).
\end{equation}
Then $\widehat{P}$ and $\widehat{Q}$  commute and satisfy the relation
$
\widehat{Q}^2 = \widehat{P}^3 - \kappa.$ Moreover, the algebra $\widehat{\gB}:= \CC[\widehat{P}, \widehat{Q}]$ has genus one and rank three. Let $q = (\lambda, \mu) \in \Spec(\widehat{\gB})$  and $\widehat{\gB} \stackrel{\chi}\lar \CC$ be the corresponding character. A straightforward computation gives the following formula for
$R_\chi := \gcd(\widehat{P}-\lambda, \widehat{Q}-\mu)$:
\begin{equation}\label{E:exRank3}
R_\chi = \partial^3 - \frac{1}{z+\mu}\partial^2 + \frac{\lambda}{z+\mu}\partial + \Bigl(\kappa + z^2 - \frac{\lambda^2}{z + \mu}\Bigr).
\end{equation}
Let $\kF$ be the spectral sheaf of $\widehat{\gB}$. The formula (\ref{E:exRank3}) yields the following result.
\begin{enumerate}
\item If $\kappa \ne 0$ then $\kF \cong \kO\bigl([q_1]\bigr) \oplus \kO\bigl([q_2]\bigr) \oplus  \kO\bigl([q_3]\bigr)$, where $q_i = (\lambda_i, 0)$ with $\lambda_i^3 = \kappa$ for  $i = 1, 2, 3$. In particular, $\det(\kF) \cong \kO\bigl(3[p]\bigr)$.
\item If $\kappa = 0$ then $\kF \cong \kE_p$, where $\kE_p$ is the indecomposable rank three vector bundle on the
cuspidal curve from Corollary \ref{C:ListRankThree}. Indeed, $\kF$ is locally free and its Fourier--Mukai transform $\kT$ is supported only at the singular point of the spectral curve due to Proposition \ref{P:SuportSpecSheaf}. Therefore, $\kF \cong \kE_q$ for some $q \in X$. From  Proposition \ref{P:continuitybundles} we deduce that $\det(\kF) \cong \kO\bigl(3[p]\bigr)$, hence $q = p$.
\end{enumerate}
\end{example}

 \subsection{Summary} Combining the classification of Gr\"unbaum \cite{Grun}, \cite[Theorem 1.2]{PrW} of Previato and Wilson with results of our article, we get the following picture. Let $\gB \subset \gD$ be a genus one and rank two commutative subalgebra. Then we have:
$$\gB = \CC[L, M] = \CC[x, y]/(y^2 - h(x)),\quad \mbox{\rm where} \quad h(x) = 4x^3 - g_2x - g_3
$$
for appropriate  parameters $g_2, g_3 \in \CC$. The operator $L$ has the form
\begin{equation*}
L = \Bigl(\partial^2 + \frac{1}{2} c_2\Bigr)^2 + \bigl(c_1 \partial + \partial c_1) + c_0\quad
\mbox{\rm and}\quad M = 2 L^{\frac{3}{2}}_+.
\end{equation*}
Let $\kF$ be the spectral sheaf of $\gB$, $\kT$ its Fourier--Mukai transform
 and $Z$ the support of $\kT$. If  the spectral curve $X = \overline{V\bigl(y^2 - h(x)\bigr)}$ is singular then $s$ denotes  its singular point. Finally, $X \stackrel{\imath}\lar X, (\lambda, \mu) \mapsto (\lambda, -\mu)$ is the canonical involution of $X$ and $p = (0:1:0)$ is the infinite point of $X$. We use the notation of Corollary \ref{C:ListRankTwo}  to describe  $\kF$.

 \smallskip
\noindent
1.~The spectral curve  $X$ is singular and $\kF \cong \kS \oplus \kS$ if and only if $c_1 = 0$  and $c_0$ is a constant. See Lemma \ref{L:selfddeg}.

\smallskip
\noindent
2.~Let $L$ be formally self--adjoint (i.e.~$c_1 = 0$) with $c'_0 \ne 0$.  Then $c_0$  and $c_2$ are given by
\begin{equation*}
c_0 = f \quad \mbox{\rm and}\quad c_2 = \frac{K_2 + 2K_3 f + f^3 - f''' f' + \frac{1}{2} (f'')^2}{f'^2}
\end{equation*}
for some  $f \in \CC\llbracket z\rrbracket$ and $K_2, K_3 \in \CC$. We have in this case: $g_2 = -2 K_3$ and $g_3 =  \frac{1}{2} K_3$.
The spectral sheaf $\kF$ is automatically locally free and self--dual. Moreover, $Z = \{q_+, q_-\} = \{(\lambda, \mu_+), (\lambda, \mu_-)\}$, where
 $\lambda = -\frac{1}{2} f(0)$ and $\mu_\pm^2 = h(\lambda)$. According to Theorem \ref{T:selfadj}, the following results are true.
 \begin{enumerate}
 \item If $q_+ \ne q_-$ then $\kF \cong \kO\bigl([q_+]\bigr) \oplus \kO\bigl([q_-]\bigr)$.
 \item If $q_+ = q_- = q$ is a smooth point of $X$ then $\kF \cong \kO\bigl([q]\bigr) \oplus
 \kO\bigl([(q)]\bigr)$ in the case $f'$ has zero of order three at $z = 0$ and $\kF \cong \kA \otimes \kO\bigl([q]\bigr)$ otherwise.
 \item If $X$ is singular and $Z = \{s\}$ then $\kF \cong \kB_p$.
 \end{enumerate}

\smallskip
\noindent
3.~Assume now that $c_1 \ne 0$, i.e.~$L$ is  not self--adjoint case. Then  $c_0, c_1$ and $c_2$ are given by
\begin{equation*}
    \left\{
    \begin{array}{ccl}
    c_0 & = & -f^2 + K_{11} f + K_{12} \\
    c_1 & = & f' \\
    c_2 & = & \dfrac{K_{14} - 2 K_{10} f + 6 K_{12} f^2 + 2 K_{11} f^3  - f^4 + f''^2 - 2 f' f'''}{2 f'^2}
    \end{array}
    \right.
\end{equation*}
where $f \in z\CC\llbracket z\rrbracket$   and $K_{10}, K_{11}, K_{12}, K_{14} \in \CC$.
The Weierstra\ss{} parameters $g_2$ and $g_3$ of the spectral curve $X$  are given by the formulae
$$g_2  = 3 K_{12}^2  + K_{10} K_{11} - K_{14} \; \mbox{\rm and}\;
g_3 = \frac{1}{4}\bigl(2 K_{10} K_{11} K_{12} + 4 K_{12}^3 + K_{14}(K_{11}^2 + 4 K_{12}) - K_{10}^2\bigr).$$
Consider the following expressions:
$$
\left\{
\begin{array}{ccl}
a(\lambda) & = & \bigl(\lambda + \frac{1}{2}K_{12}\bigr)^2 + \frac{1}{4} K_{14} \\
b(\lambda) & = & \bigl(\lambda + \frac{1}{2}K_{12}\bigr) K_{11} - \frac{1}{2} K_{10} \\
\end{array}
\right.
$$
Let $\lambda_1, \lambda_2$ be the roots of $a(\lambda)$. Then $Z = \{q_1, q_2\} = \bigl\{(\lambda_1, -b(\lambda_1)), (\lambda_2, -b(\lambda_2))\bigr\}$.
\begin{enumerate}
\item If $q_1 \ne q_2$ are smooth  then $\kF \cong \kO\bigl([q_1]\bigr) \oplus \kO\bigl([q_2]\bigr)$.
\item If $q_1 = q_2 = q$ is smooth  then $\kF \cong \kO\bigl([q]\bigr) \oplus
\kO\bigl([(q)]\bigr)$ in the case $f'$ has zero of order three at $z = 0$ and $\kF \cong \kA \otimes \kO\bigl([q]\bigr)$ otherwise, see Proposition \ref{P:exponents}.
\item The spectral curve $X$ is singular and  $Z = \{s\}$ if and only if $K_{10} = K_{14} = 0$, see Theorem \ref{T:interestingfamily}. In this case,
\begin{equation*}
X = \overline{V\Bigl(y^2  - 4\Bigl(x + \frac{K_{12}}{2}\Bigr)^2 (x - K_{12})\Bigr)}.
\end{equation*}
\begin{enumerate}
\item The  spectral sheaf $\kF$ is locally free if and only if
 $\Delta := 6 K_{12} + K_{11}^2 \ne  0$. Moreover, $\kF \cong \kB_q$ with
$q = \bigl(\frac{1}{4} K_{11}^2 + K_{12}, \frac{1}{4}K_{11}\bigl(6K_{12} + K_{11}^2\bigr)\bigr)$.
\item If $\Delta = 0$ then $\kF$ is indecomposable but not locally free. If $X$ is cuspidal (i.e.~$K_{11} = K_{12} = 0$) then $\kF \cong \kU$. If $X$ is nodal (i.e.~$K_{12} \ne 0$) then then $\kF$ is isomorphic
    to one of the sheaves $\kU_\pm$. More precisely, it is the inverse Fourier--Mukai transform of
    \begin{equation*}
\kT :=  \widehat{R}\big/\Bigl(\bigl(x + \dfrac{K_{12}}{2}\bigr)^2, y - K_{11}\bigl(x + \frac{K_{12}}{2}\bigr)\Bigr),
\end{equation*}
where $\widehat{R} := \CC\llbracket x, y\rrbracket\big/\Bigl(y^2  - 4\Bigl(x + \dfrac{K_{12}}{2}\Bigr)^2 (x - K_{12})\Bigr) \cong \widehat{\kO}_s$.
\end{enumerate}
\item The spectral curve $X$ is singular and the spectral sheaf
$\kF$ is decomposable and not locally free if and only if
$K_{10} = (3 K_{12} + \frac{1}{2} K_{11}^2) K_{11}$ and $K_{14} = - (3 K_{12} + \frac{1}{2} K_{11}^2)^2 \ne 0$.
In this case, $Z = \{s, q\}$ and $\kF \cong \kS \oplus  \kO\bigl([q]\bigr)$, where $q = \bigl(-2 K_{12} - \frac{1}{4} K_{11}^2, -\frac{1}{2}K_{11}(K_{11}^2 + 6K_{12})\bigr)$, see Theorem \ref{T:notLF}.
\end{enumerate}
\begin{remark}\label{R:Selfadjoint}
We see from this description that $L$ is non--degenerate formally self--adjoint (i.e.~$c_1 = 0$ and $c_0' \ne 0$) if and only if $\kF$ is locally free and $\det(\kF) \cong \kO\bigl(2[p]\bigr)$. For such $L$ we have: $\DD(\kF) \cong \kF$, where $\DD$ is the duality from Theorem \ref{T:FMT}. The converse is however not true.
Consider a non--self--adjoint operator $L$ with $K_{10} = K_{11} = 0$ and $K_{14} \ne 0$. Then the spectral curve $X$ is smooth (for generic $K_{14}$) and $\kF \cong \kO\bigl([q_1]\bigr) \oplus \kO\bigl([q_2]\bigr)$, where $\imath(q_i) = q_i$ for $i = 1, 2$. Therefore, $\DD(\kF)\cong \kF$ in this case.
\end{remark}

\end{document}